\documentclass{article}

\usepackage[all,arc,2cell]{xy}
\UseAllTwocells

\usepackage{amsmath}
\usepackage{amsfonts}
\usepackage{amssymb}
\usepackage{amsthm}
\usepackage{comment}
\usepackage{epsfig}
\usepackage{psfrag}
\usepackage{mathrsfs}
\usepackage{amscd}
\usepackage[all]{xy}
\usepackage{rotating}
\usepackage{lscape}
\usepackage{amsbsy}
\usepackage{verbatim}
\usepackage{moreverb}
\usepackage{sseq}
\usepackage{pdfpages}
\usepackage{enumerate}
\usepackage{array}
\usepackage{subfigure}
\usepackage{color}

\newcolumntype{C}[1]{>{\centering\let\newline\\\arraybackslash\hspace{0pt}}m{#1}}
\usepackage{float,setspace}
\restylefloat{table}

\makeatletter
\ifx\SK@label\undefined\let\SK@label\label\fi
 \let\your@thm\@thm
 \def\@thm#1#2#3{\gdef\currthmtype{#3}\your@thm{#1}{#2}{#3}}
 \def\mylabel#1{{\let\your@currentlabel\@currentlabel\def\@currentlabel
  {\currthmtype~\your@currentlabel}
 \SK@label{#1@}}\label{#1}}
 \def\myref#1{\ref{#1@}}
\makeatother

\newtheorem{theorem}{Theorem}[section]
\newtheorem{prop}[theorem]{Proposition}
\newtheorem{lemma}[theorem]{Lemma}
\newtheorem{cor}[theorem]{Corollary}
\newtheorem{conjecture}[theorem]{Conjecture}

\newtheorem{proposition}[theorem]{Proposition}
\newtheorem{thm}[theorem]{Theorem}
\newtheorem{lem}[theorem]{Lemma}

\theoremstyle{definition}
\newtheorem{definition}[theorem]{Definition}

\newtheorem{note}[theorem]{Note}
\newtheorem{defn}[theorem]{Definition}

\theoremstyle{remark}
\newtheorem{remark}[theorem]{Remark}

\newtheorem{rmk}[theorem]{Remark}

\newtheorem{example}[theorem]{Example}

\newtheorem{question}[theorem]{Question}

\newtheorem{observation}[theorem]{Observation}

\newcommand{\sO}{\mathscr{O}}

\newcommand{\sF}{\mathscr{F}}

\newcommand{\sC}{\mathscr{C}}

\newcommand{\C}{\mathscr{C}}

\newcommand{\sG}{\mathscr{G}}

\newcommand{\sL}{\mathscr{L}}

\newcommand{\sA}{\mathscr{A}}

\newcommand{\sB}{\mathscr{B}}

\newcommand{\sD}{\mathscr{D}}

\newcommand{\sP}{\mathscr{P}}

\newcommand{\sGL}{\mathscr{G}\mathscr{L}}

\DeclareMathOperator{\ox}{\otimes}
\DeclareMathOperator{\id}{id}

\DeclareMathOperator{\Mod}{Mod}

\DeclareMathOperator{\al}{\alpha}

\DeclareMathOperator{\End}{End}

\DeclareMathOperator{\Gal}{Gal}

\DeclareMathOperator{\Aut}{Aut}
\DeclareMathOperator{\Vect}{Vect}

\DeclareMathOperator{\Hom}{Hom}

\DeclareMathOperator{\Map}{Map}

\DeclareMathOperator{\Cat}{\mathrm{Cat}}
\DeclareMathOperator{\Top}{\mathrm{Top}}

\DeclareMathOperator{\Rep}{\mathrm{Rep}}
\DeclareMathOperator{\Mor}{\mathrm{Mor}}

\DeclareMathOperator{\sk}{sk}

\DeclareMathOperator{\Z}{\mathbb{Z}}

\DeclareMathOperator{\bC}{\mathbb{C}}

\DeclareMathOperator{\bK}{\mathbb{K}}

\DeclareMathOperator{\bS}{\mathbb{S}}

\DeclareMathOperator{\R}{\mathbb{R}}

\DeclareMathOperator{\iso}{\mathrm{iso}}

\newcommand{\tha}{\theta}
\newcommand{\be}{\beta}
\newcommand{\ga}{\gamma}
\newcommand{\de}{\delta}

\newcommand{\et}{\eta}
\newcommand{\io}{\iota}

\newcommand{\si}{\sigma}

\newcommand{\GA}{\Gamma}

\newcommand{\SI}{\Sigma}
\newcommand{\OM}{\Omega}

\newcommand{\PI}{\Pi}

\newcommand{\bQ}{{\mathbb{Q}}}
\newcommand{\bR}{{\mathbb{R}}}

\newcommand{\bZ}{{\mathbb{Z}}}

\newcommand{\KK}{{\bf K}}

\newcommand{\cat}{\sC\! at}
\newcommand{\tG}{\widetilde{G}}
\newcommand{\tH}{\widetilde{H}}
\newcommand{\Ch}{\sC\!at(\tilde G,}
\renewcommand{\to}{\longrightarrow}

\newcommand{\ul}{\underline}
\definecolor{orange}{rgb}{1,0.5,0}

\DeclareMathOperator{\oo}{\infty}

\newcommand{\ra}{\rightarrow}
\newcommand{\rtarr}{\rightarrow}
\newcommand{\xra}{\xrightarrow}

\usepackage[english]{babel}
\usepackage[utf8]{inputenc}

\usepackage{indentfirst}
\usepackage{comment}
 
\begin{document}

\vspace{-12mm}
\title{Equivariant algebraic K-theory of $G$-rings}

\author{Mona Merling}

\date{\vspace{1ex}}

\maketitle
\vspace{-12mm}

\begin{abstract}
A group action on the input  ring or category induces an action on the algebraic $K$-theory spectrum. However, a shortcoming of this naive approach to equivariant algebraic $K$-theory is, for example, that  the map of spectra with $G$-action induced by a $G$-map of $G$-rings is not equivariant. We define a version of equivariant algebraic $K$-theory which encodes a group action on the input in a functorial way to produce a \emph{genuine} algebraic $K$-theory $G$-spectrum for a finite group $G$. The main technical work lies in studying coherent actions on the input category. A payoff of our approach is that it builds a unifying framework for equivariant topological $K$-theory, Atiyah's Real $K$-theory, and existing statements about algebraic $K$-theory spectra with $G$-action. We recover the map from the Quillen-Lichtenbaum conjecture and the representational assembly map studied by Carlsson and interpret them from the perspective of equivariant stable homotopy theory. 
\end{abstract}


\setcounter{tocdepth}{1}
\tableofcontents

\section{Introduction}

Algebraic $K$-theory of rings is an intensely studied invariant because of its deep commections with algebraic geometry and number theory. The study of algebraic $K$-theory is linked to equivariant stable homotopy theory: one of the main computational approaches is approximating $K$-theory by topological cyclic homology, which in many cases can be computed using tools of genuine $S^1$ equivariant stable homotopy theory. However, these methods do not take into account an inherent action on the input ring (or category).
The aim of this paper is to study algebraic $K$-theory from the genuine equivariant perspective: we construct and study a genuine algebraic $K$-theory $G$-spectrum in the case when a finite group $G$ acts on the input ring.
 
Galois group actions have provided organizing principles for studying algebraic $K$-theory. It has long been suspected that the $K$-theory of a field should be computable in terms of the $K$-theory of the algebraic closure and the action of the absolute Galois group -- one of the early Quillen-Lichtenbaum conjectures was that the map from fixed points to homotopy fixed points of naive $G$-spectra is an equivalence after $p$-completion. Thomason later showed that in order to obtain an equivalence one needs to invert a ``Bott" element in $K$-theory and reduce mod a prime power. The concept of descent and the Quillen-Lichtenbaum conjecture has motivated Carlsson's program to study the $K$-theory of fields in terms of the representational assembly map for a Galois extension $E/F$ induced by tensoring a $G$-representations over $F$ with $E$. In this paper we provide the framework that allows us to interpret these maps as maps of genuine $G$-spectra or their fixed points, thus making the tools of stable equivariant homotopy theory directly available for the study of $K$-theory. 
  
  We describe our philosophy for defining equivariant algebraic $K$-theory. If the input has a $G$-action, this induces a $G$-action on the category that one builds algebraic $K$-theory out of. For example, if $R$ is a $G$-ring, then the category of finitely generated projective $R$-modules and isomorphisms $\iso \sP(R)$ has a $G$-action: for a module $M$, $gM$ is defined by twisting the scalar multiplication on $R$ by $g$. One similarly obtains a $G$-action on the category of modules over a $G$-ring spectrum $R$. 
  However, by applying the nonequivariant constructions to this category with $G$-action, we obtain just  a spectrum with $G$-action, and not a genuine $G$-spectrum -- the $K$-theory $G$-space we obtain has deloopings with respect to all spheres $S^n$ with trivial $G$-action, but it does not deloop with respect to representation spheres $S^V$. We need to modify these categories with $G$-action to turn them into ``genuine" $G$-categories, very loosely speaking. We try to make this more precise in the next paragraphs. 
 
Nonequivariantly, the algebraic $K$-theory space of $R$ is defined as the group completion of the classifying space of the symmetric monoidal category $\iso \sP(R)$, and this space is delooped using an infinite loop space machine such as  the operadic one developed by May in \cite{MayGeo} or the one based on $\Gamma$-spaces developed by Segal in \cite{segal}. These nonequivariant machines are equivalent by a celebrated theorem of May and Thomason \cite{MT}. The Segalic infinite loop space machine has been generalized equivariantly by Shimakawa in \cite{Shimakawa}, and the operadic infinite loop space machine has been generalized equivariantly by Guillou and May in \cite{GM3}, to give \emph{genuine}  $\OM$-$G$-spectra with zeroth space the group completion of the input category. We describe these machines in \S 5 where we use them, and we note that we have shown in \cite{MMO} that when fed equivalent input, they produce equivalent $G$-spectra.
But, the input these equivariant infinite loop space machines take is not simply symmetric monoidal categories with $G$-action -- their input is \emph{genuine symmetric monoidal $G$-categories}. Genuine permutative $G$-categories have been defined in \cite{GM3} as algebras over an equivariant version of the Barrat-Eccles operad, and we have defined genuine symmetric monoidal $G$-categories as pseudo algebras over the same operad in \cite{GMMO}. We will not dwell on this since all the genuine symmetric monoidal $G$-categories we consider in this paper arise in the concrete way described in the following paragraph. 
 
 \begin{definition}  Let $\tG$ be the category with objects the elements of $G$ and a unique morphism between any two objects, with $G$ acting by translation on the objects and diagonally on the morphisms. For a $G$-category $\C$, let $\cat(\tG, \C)$ be the category of all functors and all natural transformations, with $G$ acting by conjugation. \end{definition}
Note that $\tG$ is $G$-isomorphic to the translation category of $G$, and its classifying space $B\tG$ is equivalent to the total space $EG$. If $\C$ is a symmetric monoidal category with $G$-action, then it turns out that $\cat(\tG, \C)$ is an example of a genuine symmetric monoidal $G$-category, and therefore, it is input for the equivariant infinite loop space machines. We will show that replacing a symmetric monoidal category with $G$-action  $\C$ with $\cat(\tG, \C)$ not only makes it a genuine symmetric monoidal category, but it also fixes coherence issues that arise equivariantly. Even if the action does not preserve the symmetric monoidal structure strictly, but only up to coherent isomorphism, this can be rectified after applying $\cat(\tG, -)$.   
 
 We define  the equivariant algebraic $K$-theory $G$-spectrum $\KK_G(R)$ of a $G$-ring $R$ as the $\OM$-$G$-spectrum obtained by applying one of the equivariant infinite loop space machines to the category $\cat(\tG, \iso\sP(R)).$ We summarize some of the properties of $\KK_G(R)$ that we prove.
 
\begin{theorem} For finite groups $G$, the assignment $$R\mapsto \KK_G(R)$$ can be extended to a functor from $G$-rings and $G$-maps to genuine  (connective) $\Omega$-$G$ spectra, with the following properties 
\begin{enumerate} 
\item For the topological rings $\mathbb{C}$ and $\mathbb{R}$ with trivial $G$-action for any finite group $G$,  
$${\bf K_G}(\mathbb{C})\simeq ku_G \text{   and     } \  {\bf K_G}(\mathbb{R})\simeq ko_G,$$ where $ku_G$ and $ko_G$ are connective versions of equivariant topological $K$-theory;

\item For the topological ring $\mathbb{C}$ with $C_2$ conjugation action \
$${\bf K_{C_2}}(\mathbb{C})\simeq kr,$$ where $kr$ is a connective version of Atiyah's Real $K$-theory;

\item  If $|H|^{-1}\in R$, then  $$ {\bf K_G}(R)^H\simeq {\bf K}(R_H[H]),$$ where  $R_H[H]$ is the twisted group ring; 
\item For a Galois extension of rings $R\rightarrow S$  with Galois group $G$,  $${\bf K_G}(S)^G\simeq {\bf K}(R);$$

\item $\KK_G$ is invariant under a suitable notion of equivariant Morita equivalence;

\item For a finite Galois extension with group $G$, the map from fixed points to homotopy fixed points of {\bf genuine} $G$-spectra $${\bf K_G}(E)^G\to {\bf K_G}(E)^{hG}$$ is equivalent to the map from fixed points to homotopy fixed points of {\bf naive} $G$-spectra  \newline $\KK(F)\ra \KK(E)^{hG} $ from the Quillen-Lichtenbaum conjecture. 

\item For a finite Galois extension with group $G$, the representational assembly map defined by Carlsson 
\[
{\bf K}(\mathrm{Rep}_F[G])\ra {\bf K}F
\] 
is the fixed point map of a $G$-map of genuine $G$-spectra 

$${\bf K_G}(F)\to {\bf K_G}(E).$$ \end{enumerate}
\end{theorem}

 In order to deduce the first two results about topological rings, we connect the definition that we give of equivariant algebraic $K$-theory to equivariant bundle theory. Nonequivariantly,  Quillen's plus construction $BGL(R)^+$ is the zeroth component of the group completion of the monoid of classifying spaces of principal $GL_n(R)$-bundles. Equivariantly, we show that the $K$-theory space (which we define in terms of the $G$-category of projective modules) is also equivalent to the equivariant group completion of the monoid of, in this case, equivariant $GL_n(R)$-bundles. For this we use the models for equivariant bundles that we have found in \cite{GMM}. This connection allows us to recover  equivariant topological real and complex $K$-theory, and Atiyah's Real $K$-theory as examples of our construction.

For the rest of the results stated above, we need to analyze the fixed point spectrum $\KK_G(R)^H$ for subgroups $H\subseteq G$. One of the formal properties of the equivariant infinite loop space machines is that they commute with fixed points, so our task amounts to studying the fixed point categories $\cat(\tG, \C)^H$ for suitable $G$-categories $\C$. By analogy with the homotopy fixed point set of a $G$-space, we define the \emph{homotopy fixed points} of a $G$-category $\sC$ as the fixed point category $\cat(\tG, \sC)^G$, which we introduce and study in \S 2. This is the category of $G$-equivariant functors and natural transformations, which Thomason called the lax limit of the category $\C$ in \cite{homotopylimit}. However, we shift perspective: our philosophy is to work with the equivariant object $\cat(\tG, \C)$, as opposed to just restricting attention to its fixed points. We don't merely study the $H$-fixed points of  $\cat(\tG, \C)$, which are the $H$-homotopy fixed points of $\C$, but we also study how  homotopy fixed point categories relate, and for this it is convenient  to study $G$-maps between the $G$-categories  $\cat(\tG, \C)\ra \cat(\tG, \sD)$. 

In \S 4 we study the homotopy fixed point categories of module categories of $G$-rings and then we exploit these results in \S 5, \S 6, and \S 7 to draw the conclusions about the equivariant algebraic $K$-theory of $G$-rings described above. As an accidental corollary of our results about homotopy fixed points of module categories, we obtain an alternative proof of Serre's generalization of Hilbert's theorem 90, which we give in \S 6.2. The proof in the same spirit of Deligne's alternative proof of the original statement of Hilbert 90 from \cite{sga} using faithfully flat descent. 

  One property that one would expect of the homotopy fixed points of a category (which also justifies the name) is that they are homotopy invariant. We show that this is so.
  
  \begin{prop}
  A $G$-map, which is a nonequivariant equivalence of categories, induces equivalences of categories on homotopy fixed points. \end{prop}
  
  Another property of the $\cat(\tG, -)$ construction, which is more surprising maybe, and which is at the heart of our results  is that it turns maps for which equivariance holds up to isomorphism into on the nose equivariant maps. In \S 3, we define the notion of a pseudo equivariant functor between $G$-categories as a functor which commutes with the $G$-action only up to coherent isomorphism. Very precisely, if one regards $G$-categories as functors $G\ra \cat$, then an equivariant map of $G$-categories is a natural transformation between these and a pseudo equivariant map is a pseudo natural transformation. The main result of that section is as follows.
  \begin{prop}\mylabel{pseudo intro}
 Given a pseudo equivariant functor of $G$-categories $\C\ra \sD$, there is an induced \emph{on the nose equivariant} functor  $\cat(\tG, \C) \ra \cat(\tG, \sD)$, so there are  induced maps on homotopy fixed point categories $\C^{hH}\ra \sD^{hH}$ for all subgroups $H$ of $G.$ \end{prop}

We showcase some of the applications of the result about pseudo equivariant functors. For example, the extension of scalars map between the module categories of $G$-rings (with actions defined a few paragraphs above) along a $G$-map of rings, is not equivariant, but only pseudo equivariant. Because this  allows us to construct  an on the nose equivariant functor after applying $\cat(\tG, -)$ to our module categories, we can ensure that we actually get a functor from the category of $G$-rings to the category of $G$-spectra. The definition of equivariant Morita equivalence given in \S \ref{morita}, which equivariant algebraic $K$-theory is invariant under, is also in terms of a pseudo equivariant functor.  We have claimed above that applying $\cat(\tG, -)$ rectifies an action that preserves the symmetric monoidal structure of a category $\C$ to an action that preserves it strictly. This is also an application of the same result: the functor \linebreak $\C\times \C \ra \C$ that gives the symmetric monoidal structure is pseudo equivariant. 


In upcoming work with C. Malkiewich \cite{CaryMona} we extend this work on equivariant algebraic $K$-theory -- we define and study equivariant $A$-theory, and \myref{pseudo intro} is  essential for studying Waldhausen $G$-categories.
 The fixed point categories of a Waldhausen category $\C^{H}$ are not Waldhausen, because the action does not preserve the zero object or the pushouts strictly, and the result that a pseudo equivariant functor can be strictified to an equivariant functor after applying $\cat(\tG, -)$ allows us to show that the homotopy fixed point categories $\C^{hH}$ \emph{are} Waldhausen categories. Moreover, using the \myref{pseudo intro} we show in \cite{CaryMona} how one gets transfer ``wrong way" maps between the homotopy fixed point categories, and use them to construct  one version of equivariant $A$-theory as a ``spectral Mackey functor." We point out that C. Barwick has also given a definition of equivariant $K$-theory  in \cite{Gmonster} using ``spectral Mackey functors," in the setting of $\infty$-categories, and we hope we will be able to give a comparison in the future.

We conclude the introduction with two technical remarks. Note that everywhere we need to take classifying spaces of categories that are clearly not small. Nonequivariantly, it is always assumed in $K$-theory that when we take the classifying space of a category which is not small, such as $\sP(R)$, $\sF(R)$, or $\Mod(R)$, we are tacitly replacing the category by a small category, which is equivalent to it, such as its skeleton. The situation is a little trickier equivariantly, because we do not have an equivariant equivalence between a $G$-category and its skeleton. This is too much to hope for; however, we show in \S \ref{skeleton} that there is a weak $G$-equivalence $\cat(\tG, \iso\sF(R)) \simeq \cat(\tG, \sGL(R))$. In \S \ref{skeleta} we generalize this to any $G$-category -- this is even more subtle  because unlike in the case of free modules where we show that $R^n\cong gR^n$, in general, an object $C$ is not necessarily isomorphic to $gC$. We show that for a $G$-category $\C$, we can put a $G$-action on the skeletal category $\sk \C$, such that we get a weak $G$-equivalence  $\cat(\tG, \C)\to \cat(\tG, \sk \C)$. This suffices for our applications, because in equivariant algebraic $K$-theory we are only taking classifying spaces of $G$-categories of the form $\cat(\tG, \C)$.

We end with a remark about the group $G$. All of our categorical work on homotopy fixed points of $G$-categories works for any topological group $G$. However, in order to pass the statements to the spectrum level, we have to restrict to finite groups because of the limitations of the equivariant infinite loop space machines. We do hope that these limitations can be overcome in the near future, at least for profinite groups.

\subsection*{Acknowledgements}
We would like to thank Peter May for suggesting equivariant algebraic $K$-theory as a thesis topic; Bertrand Guillou, Peter May, and Ang{\'e}lica Osorno for all the joint work that has fed into this project;  Ang{\'e}lica Osorno  for adding categorical sophistication to some of this work; 
Andrew Blumberg, Lars Hesselholt and Mike Hill, for sharing their invaluable insights on innumerable occasions; Dan Grayson, for sharing his intuition and explaining alternative approaches to some of our results,  and Clark Barwick for sharing his perspective on equivariant algebraic $K$-theory and  for the many discussions about how our approaches compare. We would also like to thank Emanuele Dotto, Cary Malkiewich, Lennart Meier, Daniel Sch\"{a}ppi and Aaron Royer for many great conversations on the topics of this paper and related work in progress on equivariant $K$-theory for Waldhausen $G$-categories.

 \section{ Homotopy fixed points of a category}
 
By analogy with the homotopy fixed points for a $G$-space we define the homotopy fixed points of a $G$-category. These were also studied by Thomason under the name \emph{``lax limit"} in \cite{homotopylimit}. However, we take an equivariant point of view: for us, the homotopy fixed points are the actual fixed points of a $G$-category, and we study this equivariant object as opposed to just restricting attention to the fixed points. 
  
  \subsection{Preliminaries on $G$-categories}
Concisely, a $G$-category can be defined as a functor $G\ra \cat$. Explicitly, the data of such a functor is a category $\sC$, and for each $g\in G$, an endofunctor
$(g\cdot) \colon \sC\ra \sC$ such that $(e\cdot) = \id_\sC$ and  $(g\cdot) \circ (h\cdot) = (gh)\cdot$. By slight abuse, we will often call the category $\C$ a $G$-category, which means we are implicitly thinking of the action endofunctors $(g\cdot)$. Sometimes  we might omit the ``$\cdot$" from the notation and write $gC$ or $gf$ to denote the action of $g$ on an object $C$ or a morphism $f$. 
A natural transformation of functors $G\ra \cat$ translates to a functor between the two $G$-categories which commutes with the $G$-action. We denote the category of $G$-categories and $G$-equivariant functors by $G\cat$.
  
For subgroups $H\subseteq G$, we define the $H$-fixed point category $\sC^H$  of a $G$-category $\C$ as  the subcategory  with objects those $C \in \sC$ such that
$hC = C$ and morphisms those $f\in\sC$  such that $hf = f$ for all $h\in H$. This definition coincides with the categorical definition as $\lim_H \C$ when we think of $\C$ as a functor $G\ra \cat$. A crucial fact is that the classifying space functor $B\colon \cat\ra \Top$ commutes with fixed points, namely 
\begin{equation}\label{B fixed points} B(\C^H) =(B\C)^H.\end{equation} 

\begin{definition}\mylabel{weak equiv}
A functor between $G$-categories $F\colon \C\ra \sD$ is \emph{a weak $G$-equivalence} if  it induces a weak $G$-equivalence on classifying spaces $BF\colon B\C \ra B\sD.$ \end{definition}

 \subsection{$G\cat$ as a 2-category}\label{2cat}

 We may  view $\cat$ as the $2$-category of categories, with $0$-cells, $1$-cells,
and $2$-cells the categories, functors, and natural transformations.  From
that point of view, $\cat$ is enriched over itself: the internal hom, $\cat(\sA,\sB)$, is the category whose objects are the functors $\sA\rtarr \sB$ and whose morphisms are the natural transformations 
between them.

 Similarly, we may view $G\cat$ as the underlying $2$-category of a category enriched 
over $G\cat$. The $0$-cells are $G$-categories, and the internal hom between them is the
$G$-category $\cat(\sA,\sB)$. Its 
underlying category is $\cat(\sA,\sB)$, and $G$ acts by conjugation on 
functors and natural transformations.  Thus, for $F\colon \sA\rtarr \sB$, 
$g\in G$, and $A$ either an object or a morphism of $\sA$, 
$(gF)(A) = gF(g^{-1}A)$.  Similarly, for a natural transformation 
$\et\colon E\rtarr F$ and an object $A$ of 
$\sA$, 
$$(g\et)_A = g\et_{g^{-1}A}\colon g E(g^{-1}A)\rtarr gF(g^{-1}A).$$
The category $G\cat(\sA,\sB)$ of $G$-equivariant functors and $G$-equivariant natural transformations is the same as the $G$-fixed category 
$\cat(\sA,\sB)^G$.

\begin{rmk}\label{topological}
We can topologize the definitions so far, starting with the $2$-category of categories 
internal to the category $\Top$, together with continuous 
functors and continuous natural transformations.  A \emph{topological $G$-category} $\sA$ is a category  internal to the cartesian monoidal category $G\Top$. It has object and morphism $G$-spaces and continuous $G$-equivariant source, target, identity and composition structure maps respecting the usual category axioms. 
These are more general than (small) topologically enriched categories, which have discrete sets of objects. 
\end{rmk}

 \subsection{The functor $\cat(\tG,-)$ and  homotopy fixed points of categories}\label{catg}

\begin{definition}
For a topological group $G$, define $\tG$ to be the topological $G$-groupoid with object space  $G$ and morphism space $G\times G$. The source and target maps are the projections onto the two factors.\end{definition}
 Thus the objects of $\tG$ are the elements of $G$ and there is a unique morphism between any two objects. We choose to label the unique morphism $g\rightarrow h$ by the pair $(h,g)$ in order to be consistent with \cite{GMM}. The idea is that reversing the order of source and target makes the notation for composition more transparent: $(g,h)\circ (h, k)=(g, k).$ The $G$-action on $\tG$ is given by translation on the objects, which forces it to be diagonal on morphisms, since $g(h\ra k)$, namely $g(k,h)$ must be the unique map $gh \ra gk$, namely $(gk, gh)$. 

\begin{definition} Define the translation category of $\ul{G}$ of $G$ in the standard way as having object space $G$ and morphism space $G\times G$, with the morphism $h\ra gh$ labeled by $(g, h)$.\end{definition}

Again, since there is a unique morphism between any two objects, the $G$-action on objects by translation completely determines the action on the morphism space: $G$ acts on the second coordinate of $G\times G$. The following lemma follows immediately from the fact that  $G\times G$ with $G$ acting diagonally and $G\times G$ with $G$ acting on the second coordinate are $G$-homeomorphic.

\begin{lemma}
The translation category $\ul{G}$ is $G$-isomorphic to the category  $\tG$.
\end{lemma}

\begin{rmk}\mylabel{chaotic}
The category $\tG$ is an instance of the more general concept of \emph{chaotic category} corresponding to a space. There is a \emph{chaotic category functor} from spaces to categories (actually, to groupoids), sending a space $X$ to the category $\widetilde{X}$ with object space $X$ and morphism space $X\times X$; there is a unique morphism between any two objects in $\widetilde{X}$. The relevant point is that the object functor is right adjoint to the chaotic category functor, and in particular, we have a homeomorphism between the mapping spaces \begin{equation}\label{adjunction}
\cat (\sC,\widetilde{X})\cong \Map(\sO\!b\sC,X).\end{equation}

Similarly, the translation category $\ul{G}$ of $G$ is an instance of the more general notion of translation category of a $G$-space. For a $G$-set, or a $G$-space $X$, we denote by $\underline{X}$ the translation category of $X$ with objects the points of $X$ and morphisms $(g, x): x\rightarrow gx$. However, as we have seen, the concepts of chaotic and translation category agree for $G$ up to $G$-isomorphism. Thus, it is harmless to think of $\tG$ as the translation category of $G$. For a more comprehensive treatment of both chaotic and translation categories, we refer the reader to \cite{GMM}.
\end{rmk} 

We make the following crucial observation.

\begin{observation}\mylabel{EG}
The classifying space  $B\tG$ is $G$-equivalent to the universal principal $G$-bundle $EG$ since $\tG$ is a contractible category (every object is initial and terminal) and it has a free $G$-action.
\end{observation}

We have a functor $\cat(\tG, -)$ from $G$-categories to $G$-categories, which sends a $G$-category $\sC$ to the category of functors and natural transformations $\cat(\tG, \C)$, with $G$-action by conjugation, as described in section \ref{2cat}. This is a topological category when $\sC$ is such.  In view of Section \ref{2cat}, $\cat(\tG,-)$ can be viewed as a 2-functor. Observe that the  functor 
$\cat(\tG,-)$  is corepresented and is thus a right 
adjoint. Therefore it preserves all limits; in particular it preserves products, which will be crucial to our applications.

The equivariant projection $\tG\rtarr \ast$ to 
the trivial $G$-category induces a natural $G$-map 
\begin{equation}\label{iota}
\io\colon  \sA \simeq \cat(\ast,\sA) \rtarr \cat(\tG,\sA),
\end{equation}
which is always a nonequivariant equivalence of $G$-categories, but not usually a $G$-equivalence.
However, as observed in \cite{GM3},
the functor $\cat(\tG,-)$ is idempotent:

\begin{lem}\mylabel{idem} For any $G$-category $\sA$, 
\[ \io\colon \cat(\tG,\sA) \rtarr  
\cat(\tG,\cat(\tG,\sA))  \]
is an equivalence of $G$-categories.
\end{lem}

By analogy with the definition of homotopy fixed points of $G$-spaces, we make the following definition.

\begin{definition}
The \emph{homotopy fixed points} of a $G$-category $\C$, denoted by $\C^{hG},$  are defined as $\cat(\tG, \C)^G$, namely the $G$-equivariant functors $\tG\ra \C$ and the $G$-natural transformations between these.
\end{definition}

\begin{observation}\mylabel{hfixedpoints}
Note that $\tH$ and $\tG$ are equivalent as $H$-categories since they are both $H$-free contractible categories. Therefore we can identify $$\cat(\tG, \C)^H = H\cat(\tG, \C)) \simeq H\cat(\tH, \C)=\cat(\tH, \C)^H.$$
Consequently, for any $H\subseteq G$, we can unambiguously define the $H$-homotopy fixed points of a $G$-category  $\C$ as either $\cat(\tG, \C)^H$ or by applying the definition above of homotopy fixed points to $\sC$, regarded as an $H$-category. And conveniently, for any statement that we wish to prove holds for $H$-fixed points $\cat(\tG, \sC)^H$ for any $H\subseteq G$, it is enough to prove it for $G$-fixed points $\cat(\tG, \sC)^G$ as long as $G$ is arbitrary.

\end{observation}

 \subsection{Explicit description of homopy fixed point categories}\label{fixed points}

We describe explicitly the category of equivariant functors and equivariant
 natural transformations $G\cat(\tG, \C)$.  Any $G$-fixed functor $F:\tG \rightarrow \C$ is determined on objects by where the identity $e$ of $G$ gets mapped to since $F(g)=g\cdot F(e)$. On morphisms, $F$ is determined by where it sends morphisms of the type $(g,e)$ since $F(g,h)=h\cdot F(h^{-1}g, e)$. We have that $F(e,e)=id_C$, where $id_C$ is the identity morphism of the object $C\in \C$ and $F(e)=C$. The following cocycle condition is also satisfied: $$F(gh, e)=F(gh,g)F(g,e)=g\cdot F(h,e) F(g,e).$$

We summarize this discussion in the following result, which gives an explicit description of the homotopy fixed point category of a $G$-category $\C$.\footnote{This explicit description is also given in  more concise terms in \cite{homotopylimit}.}

\begin{prop}\mylabel{wrong homotopy fixed points}
 The objects of the homotopy fixed point category $\C^{hG}=G\cat(\tG, \C)$ are pairs $(C,f)$ where $C$ is an object of $\C$ and $f:G \rightarrow \Mor(\C)$ is a map from $G$ to morphisms of $\C$ such that $f(g):C \rightarrow g\cdot C$ and $f$ satisfies the condition  $f(e)=id_C$ and the cocycle  condition
\begin{equation}\mylabel{wrong cocycle}
 f(gh)=(g\cdot f(h)) f(g).
\end{equation}

A morphism $(C,f)\rightarrow (C', f')$ is given by a morphism  $\alpha: C\rightarrow  C'$ in $\C$ such that the following diagram commutes for any $g\in G$ :
$$ \xymatrix{
C \ar[d]_\alpha \ar[rr]^{f(g)} && g\cdot C \ar[d]^{g\cdot \alpha}\\
C' \ar[rr]^{f'(g)} && g\cdot C'} $$
\end{prop}

However, the alternative cocycle condition \begin{equation}\mylabel{cocycle} f(gh)=f(g)(g\cdot f(h)) \end{equation} is the standard one, which will appear in all of our applications. For instance, this is the condition that yields a crossed homomorphism when $\sC$ is a group, whereas condition (\ref{wrong cocycle}) yields a crossed antihomomorphism, which is less customary.

We show that changing condition (\ref{wrong cocycle}) to the usual cocycle condition  (\ref{cocycle}) is inoffensive since it yields an isomorphic category. The proof is a straightforward generalization of the proof of Lemma 4.13 in \cite{GMM}, which is the special case where the category $\C$ is a group. 
\begin{prop}\mylabel{homotopy fixed points}
 
 There is an isomorphism of categories between the homotopy fixed point category $\C^{hG}=G\cat(\tG, \C)$ and the category described as follows.
 The objects  are pairs $(C,f)$ where $C$ is an object of $\C$ and $f:G \rightarrow \Mor(\C)$ is a map from $G$ to morphisms of $\C$ such that $f(g):g\cdot C \rightarrow C$ and $f$ satisfies the condition  $f(e)=id_C$ and the cocycle  condition
$$f(gh)=f(g)(g\cdot f(h)). $$

A morphism $(C,f)\rightarrow (C', f')$ is given by a morphism  $\alpha: C\rightarrow  C'$ in $\C$ such that  the following diagram commutes for any $g\in G$:
$$ \xymatrix{
g\cdot C \ar[d]_{g\cdot \alpha} \ar[rr]^{f(g)} &&  C \ar[d]^{\alpha}\\
g\cdot C' \ar[rr]^{f'(g)} &&  C'} $$
\end{prop}

\begin{proof} We explicitly construct the isomorphism between the category described in \myref{homotopy fixed points} to the category described in \myref{wrong homotopy fixed points}. The construction of the inverse isomorphism is similar.

Let $f\colon G\rtarr \Mor(\sC)$ be such that $f(g):g\cdot C \rightarrow C$ and suppose that $f$ satisfies $f(e)=id_C$ and condition (\ref{cocycle}). Note that $f(g)$ is an isomorphism with inverse $g\cdot f(g^{-1})$ for all $g\in G$.

Define 
$\bar{f}\colon G\rtarr \Mor(\C)$ by 
$$\bar{f}(g) = g\cdot f(g^{-1}),$$ so that $\bar{f}\colon C\ra g\cdot  C$.  Then
\[ \bar{f}(gh) =  (gh)\cdot f(h^{-1}g^{-1}) 
= g\cdot h\cdot(f(h^{-1})(h^{-1}\cdot f(g^{-1})) 
= (g\cdot \bar{f}(h))(\bar{f}(g)),\]
so that $\bar{f}$ is satisfies condition  (\ref{wrong cocycle}). 

 Let $f,f'\colon G\rtarr \Mor(\sC)$, such that $f(g):g\cdot  C \rightarrow  C$ and $f(g):g\cdot C' \rightarrow C'$ Suppose that $f$ and $f'$ satisfy  (\ref{cocycle}), and $\al:C\ra C'$ is a morphism in $\sC$ for which the diagram 
$$ \xymatrix{
g\cdot C \ar[d]_{g\cdot \alpha} \ar[rr]^{f(g)} &&  C \ar[d]^{\alpha}\\
g\cdot C' \ar[rr]^{f'(g)} &&  C'} $$commutes for all $g\in G$. Then
\[ \bar{f'}(g)\al = (g\cdot f'(g^{-1}))\al
= g\cdot (f'(g^{-1})(g^{-1}\cdot\al)) = g\cdot (\al f(g^{-1})) 
= (g\cdot\al)\bar{f}(g),\]
i.e. the diagram
$$ \xymatrix{
C \ar[d]_\alpha \ar[rr]^{\bar{f}(g)} && g\cdot C \ar[d]^{g\cdot \alpha}\\
C' \ar[rr]^{\bar{f'}(g)} && g\cdot C'} $$
commutes  for all $g\in G$.
\end{proof}

\begin{remark}\mylabel{trivial action}
Note that if $G$ acts trivially on $\C$, then $\cat(\tG,\C)^G\cong \cat(G, \C)$, the category of functors $G\rightarrow \C$,  i.e., $\C^{hG}$ is the category of representations of $G$ in $\C$.
 \end{remark}

We emphasize that the homotopy fixed points do not in general commute with the classifying space functor. However, if  $\C$ is a \emph{discrete} $G$-groupoid, then the comparison map $B\cat(\tG, \sC)\ra \Map(EG, B\C)$ is a weak $G$-equivalence, and we have $$B(\sC^{hH})\simeq (B\C)^{hH}$$ for any $H\subseteq G$ (see \cite[\S 5]{GMM} for a proof). However, in some of our examples of interest, even when the category $\C$ is a groupoid, we need it to be topological, and therefore cannot assume this commutation.

\subsection{Homotopy fixed points of a group}\label{group case}
 The homotopy fixed point category $ \C^{hG}$ simplifies when $\C= \Pi$, a topological group regarded as a topological category with one object, with $G$-action.  In that case, the homotopy fixed points can be interpreted in terms of the well-known notion of \emph{crossed group homomorphisms.} The category $G\cat(\tG, \Pi)$ has been studied extensively in \cite{GMM}, where it was shown that it  has the following interpretation.
 
 \begin{thm}[\cite{GMM}, 4.15]\mylabel{crossedhom}
 
Suppose $\Pi$ is a group with $G$-action. The homotopy fixed point category $\Pi^{hG} $ is equivalent to the \emph{crossed functor 
category} $\cat_\times (G,\PI )$ whose objects are the continuous crossed 
homomorphisms $G\rtarr \PI$ and whose morphisms $\si\colon \al\rtarr \be$ are the elements 
$\si\in \PI$ such that $$ \be(g)(g\cdot\si) = \si\al(g).$$ 
\end{thm}

 \noindent This interpretation leads to the following condition for when the nonequivariant equivalence
$$\io\colon  \Pi \simeq \cat(\ast,\Pi) \rtarr \cat(\tG,\Pi)$$ is a weak $G$-equivalence.  
\begin{prop}[\cite{GMM}, 4.19]\mylabel{H1 trivial}
The functor $\io^H\colon \PI^H \rtarr \Ch \PI)^H$
is an equivalence of categories if and only if the first nonabelian cohomology set
$H^1(H;\PI)$ is trivial.
\end{prop}

\subsection{Homotopy invariance of homotopy fixed points}
The following lemma is inspired by the analogous result for homotopy fixed points of $G$-spaces or naive $G$-spectra. 
 \begin{proposition}\mylabel{homotopy fixed point weak equiv}
 If $\Theta \colon \sC\to \sD$ is a $G$-functor that is a nonequivariant equivalence of categories then the functor induced by post composition
 $$\cat(\tG, \sC)\to \cat(\tG, \sD)$$ is a weak $G$-equivalence of categories.
 \end{proposition}

 \begin{proof}
 From \myref{EG}, we see that it is enough to prove that we get an equivalence on $G$-fixed points. The map $$ \cat(\tG, \sC)^G \to \cat(\tG, \sD)^G$$ is faithful since it is the restriction of a faithful map to a subcategory. We show that is essentially surjective and full. We use the explicit description of fixed points given above.\\
 
 Pick an object $(D, f)$ in $\cat(\tG, \sD)^G$. Since $\Theta$ is essentially surjective, there exists a nonequivariant isomorphism $\psi \colon D\xrightarrow{\cong} \Theta(C)$ for some $C\in \sC$. By applying $g\cdot$ we get  $g\psi\colon gD\xrightarrow{\cong} g\Theta(C)=\Theta(gC)$. Since $\Theta$ is fully faithful, for every $f(g)\colon D\xrightarrow{\cong} gD$ there exists a unique map $f'(g)\colon C \xrightarrow{\cong} gC$ such that $\Theta(f'(g))$ is the composite
 \[ \Theta(C)\xrightarrow{\psi^{-1}} D\xrightarrow{f(g)}gD\xrightarrow{g\psi} g\Theta(C),\] and $f'(g)$ is an isomorphism since $f(g)$ and $\psi$ are. We need to check the cocyle condition on $f'$.  
 
 We will read it off from the following commutative diagram 
 \[ \xymatrix{
 D \ar[d]_-\psi \ar[rr]^-{f(g)} && gD \ar[d]^-{g\psi} \ar[rr]^-{gf(h)} && gh D\ar[d]^-{gh\psi}\\
 \Theta(C) \ar[rr]^-{\Theta(f'(g))} && g\Theta(C) \ar[rr]^-{g\Theta(f'(h))} && gh \Theta(C)
 }\]
  The top composite is $f(gh)$ since $f$ satisfies the cocycle condition. Thus the bottom map must be $\Theta(f(gh))$. By the commutation of $g$ with $\Theta$, the bottom map is just $\Theta$ applied to the composite $$C\xrightarrow{f'(g)} gC \xrightarrow{gf'(h)} ghC.$$ 
Thus $f'$ satisfies the cocycle condition. 

We are left to show fullness. Suppose we have a morphism in $\cat(\tG, \sD)$ from $(\Theta(C), \Theta(f))$ to $(\Theta(C'), \Theta(f'))$ given by the diagrams
 \[ \xymatrix{
 \Theta(C)\ar[d]_-{\al} \ar[rr]^-{\Theta(f(g))} && g\Theta(C) \ar[d]^-{g\al} \\
 \Theta(C') \ar[rr]^-{\Theta(f'(g))} && g\Theta(C') 
 }\]
Since $\Theta$ is full, there exists a map $C\xrightarrow{\al'} C'$ such that $\Theta(\al')=\al$. Thus there is a map in $\cat(\tG, \sC)$ 
 \[ \xymatrix{
C\ar[d]_-{\al'} \ar[rr]^-{f(g)} && gC \ar[d]^-{g\al'} \\
 C' \ar[rr]^-{f'(g)} && gC' 
 }\]
 whose image is the map above. This gives fullness.
  \end{proof}

\section{Pseudo equivariance}\label{pseudo}
 
 \subsection{Pseudo equivariant functors}\label{pseudo1}

Let $\sC$ and $\sD$ be $G$-categories. We define the notion of a $\emph{pseudo equivariant}$ functor  $\Theta\colon \sC\ra \sD$, and  we then show that such a functor induces an on the nose equivariant functor $$\cat(\tG, \sC) \to \cat(\tG, \sD).$$
Thus it induces maps on fixed points $$\cat(\tG, \sC)^H \to \cat(\tG, \sD)^H$$ for all subgroups $H\subseteq G$. Pseudo equivariance will be absolutely central to a lot of our $K$-theory results because often enough the maps between the $G$-categories we will consider in equivariant algebraic $K$-theory are not on the nose equivariant, but pseudo equivariant. The construction of equivariant algebraic $K$-theory will ensure that this is enough to get actual equivariant maps on the spectrum level. Moreover, this result will be what allows us to rectify a $G$-action on a symmetric monoidal or Waldhausen category which does not preserve the structure strictly, but only up to coherent isomorphism.

\begin{definition}\mylabel{weakequivariant}
A $\emph{pseudo equivariant}$ functor between $G$-categories $\sC$ and $\sD$ is a functor $\Theta \colon \sC \ra \sD$, together with
 natural isomorphisms of functors  $\theta_g$ for all $g\in G$
 
 \[
\xymatrix{
\ \C\  \ar[r]^-{g\cdot} \ar[d]_{\Theta} \drtwocell<\omit>{<0>\quad   \theta_g}  & \ \sC\  \ar[d]^{\Theta}  \\ 
\ \sD\ \ar[r]_-{g\cdot} & \ \sD\ . \\}
\] 
 such that  $\theta_e=\id$ and  for $g,h \in G$ we have an equality of natural transformations, where on the left hand side we are considering the composite of natural transformations.
\[
\xymatrix{
\ \C\  \ar[r]^-{h\cdot} \ar[d]_{\Theta} \drtwocell<\omit>{<0>\quad   \theta_h}  & \ \sC\  \ar[d]^{\Theta} \  \ar[r]^-{g\cdot}  \drtwocell<\omit>{<0>\quad   \theta_g}  & \ \sC\  \ar[d]^{\Theta}  \\ 
\ \sD\ \ar[r]_-{h\cdot} & \ \sD\ \ar[r]_-{g\cdot} & \ \sD\ } 
\qquad = \qquad
\xymatrix{
\ \C\  \ar[r]^-{gh\cdot} \ar[d]_{\Theta} \drtwocell<\omit>{<0>\quad   \theta_{gh}}  & \ \sC\  \ar[d]^{\Theta}  \\ 
\ \sD\ \ar[r]_-{gh\cdot} & \ \sD\  }
\]
Note that requiring this equality makes sense because the outer right down and down right composites in the two diagrams are equal. Explicitly, for $C$ an object of $\C$, this means that the following diagram commutes:

\[\xymatrix{
\Theta(ghC) \ar[rr]^{\theta_g (hC)} \ar@/^3pc/[rrrr]^{\theta_{gh}(C) \ \ }  & & g\Theta(hC) \ar[rr]^{g \theta_h(C)}  & & gh \Theta(C)  
} \]

\end{definition}

\begin{remark}\mylabel{afterthetadiagram}
If $\theta_g$ are equalities for all $g\in G$, then $\Theta$ is actually an equivariant functor.
\end{remark}

We explain the choice of nomenclature. Recall that a $G$-category is defined by a functor \newline $G\ra \cat$, and an equivariant map between $G$-categories is then just a natural transformations of such functors. A \emph{pseudo equivariant} map between $G$-categories is a \emph{pseudo natural transformation}. 
We prove next that a pseudo equivariant functor $\Theta \colon \sC \ra \sD$ naturally induces an on the nose equivariant map after applying the $\cat(\tG, -)$ functor. 

\begin{proposition}\mylabel{pseudo equiv}
A pseudo equivariant functor $\Theta \colon \sC \ra \sD$ naturally induces an equivariant functor $$\tilde{\Theta}\colon \cat(\tG, \sC)\to \cat(\tG, \sD).$$
\end{proposition} 

\begin{proof}
Clearly post composing a functor $F\colon \tG\to \sC$ with $\Theta$ does not yield an equivariant functor, but we can use the natural isomorphisms $\theta_g$ to create one. We define
$$\tilde{\Theta}(F)(g)=g\Theta ((g^{-1}F)(e))=g \Theta(g^{-1} F(g))$$

Recall that there is a unique map in $\tG$ from $g$ to $g'$, which we denote by $(g',g)$. Applying $\Theta\circ F$ we get a map  $\Theta(F(g))\xrightarrow{\Theta(F(g', g))} \Theta(F(g'))$ in $\sD$.  We define $\tilde{\Theta}(g',g)$ to be the composite
$$ g \Theta(g^{-1} F(g)) \xrightarrow[\cong]{\theta_g^{-1}}  \Theta(gg^{-1} F(g))  \xrightarrow{\Theta(F(g',g))}  \Theta(g'g'^{-1} F(g'))  \xrightarrow[\cong]{\theta_{g'}} g' \Theta(g'^{-1} F(g')).$$

For a morphism in $\cat(\tG, \sC)$, namely a natural transformation $\eta\colon F\Rightarrow  E$, we define the components of $\tilde{\Theta}(\eta)$ as $$\tilde{\Theta}(\eta)_g= g\Theta(g^{-1}\eta_g).$$   In order to see that this is indeed a natural transformation $\Tilde{\Theta}(F)\Rightarrow \Tilde{\Theta}(E)$, note that the naturality square for $\tilde{\Theta}(\eta)_g$ translates to 
\[ \xymatrix{
g\Theta(g^{-1}F(g)) \ar[dd]_-{\theta^{-1}_{g}} \ar[rr]^-{g\Theta(g^{-1}\eta_{g})} && g\Theta(g^{-1} E(g))\ar[dd]^-{\theta^{-1}_{g}}\\
&&&&\\
\Theta(gg^{-1} F(g)) \ar[dd]_-{\Theta(F(g',g))} \ar[rr]^-{\Theta(\eta_g)} && \Theta(gg^{-1} E(g))\ar[dd]^-{\Theta(E(g',g))} \\
&&&&\\
\Theta(g'g'^{-1} F(g')) \ar[dd]_-{\theta_{g'}} \ar[rr]^-{\Theta(\eta_{g'})} && \Theta(g'g'^{-1} E(g')) \ar[dd]^-{\theta_{g'}}\\
&&&&\\
g'\Theta(g'^{-1} F(g')) \ar[rr]^-{g'\Theta(g'^{-1} \eta_{g'})} && g' \Theta(g'^{-1}E(g'))
}\]
and all the squares commute by the naturality of $\eta$ and of the $\theta_g$'s.


We check that with these definitions $\tilde{\Theta}$ is indeed an equivariant functor.  Let $F$ be an object of $\cat(\tG, \C)$. We need to check that for $h\in G$, $\tilde{\Theta}(hF)=(h\tilde{\Theta}(F))$. On objects, it is not hard to see that the two functors  agree:
$$ \tilde{\Theta}(hF)(g)=g\Theta(g^{-1}hF(h^{-1}g))=hh^{-1}g\Theta(g^{-1}hF(h^{-1}g))=h\tilde{\Theta}(F)(h^{-1}g)=(h\tilde{\Theta}(F))(g).$$

On morphisms $(g', g)$ they agree by the commutativity of the following diagram where the top row is $\tilde{\Theta}(hF)(g', g)$ and the bottom row is $(h\tilde{\Theta}(F))(g',g)$. To avoid cluttering the diagram, we denote by $f$ the morphism  $F(h^{-1}g', h^{-1}g)$ in $\C$.

\begin{center}
\[
\xymatrix{
g\Theta(g^{-1}hF(h^{-1}g)) \ar[dd]_= \ar[rr]^-{\theta^{-1}_g} && \Theta(hF(h^{-1}g)) \ar[dd]_-{\theta_h} \ar[rr]^-{\Theta(hf)}&&  \Theta( hF(h^{-1}g')) \ar[rr]^-{\theta_{g'}} \ar[dd]^-{\theta_h} &&g'\Theta(g'^{-1}hF(h^{-1}g')) \ar[dd]^=\\
&&&&&&&&\\
hh^{-1}g\Theta(g^{-1}hF(h^{-1}g)) \ar[rr]_-{h\theta^{-1}_{h^{-1}g}} && h \Theta(F(h^{-1}g))  \ar[rr]^-{h\Theta(f)}&& h \Theta( F(h^{-1}g')) \ar[rr]_-{h\theta_{h^{-1}g'}}  &&hh^{-1}g'\Theta(g'^{-1}hF(h^{-1}g')) 
}\]
\end{center}

The center square is just the naturality square for the natural transformation $\theta_h$. The right square is the diagram from the compatibility we have required of the $\theta_g$'s as expressed in the diagram above \myref{afterthetadiagram}, and it is not hard to check that the commutativity of the first square also follows from the same compatibility condition translated in terms of inverses. 

On natural transformations it is again not hard to check that $\tilde{\Theta}$ is equivariant:
$$\tilde{\Theta}(h\eta)_g= g\Theta(g^{-1}(h\eta)_g)=hh^{-1}g\Theta(g^{-1}h\eta_{h^{-1}g})=h\tilde{\Theta}(\eta)_{h^{-1}g}=(h\tilde{\Theta}(\eta))_g.$$
\end{proof}

\begin{cor}\mylabel{pseudo equiv cor}
 A pseudo equivariant functor $\Theta \colon \sC \ra \sD$, induces functors  $\tilde{\Theta}^H\colon \C^{hH}\ra \sD^{hH}$  on homotopy fixed points for all $H\subseteq G$.
\end{cor}

Note that the definition of $\tilde{\Theta}$ on objects makes sense for any functor $\Theta$ and $\tilde{\Theta}$  is equivariant on objects. However, without the isomorphisms encoded in the pseudo equivariance condition for $\Theta$, it is not possible to define the map $\tilde{\Theta}$ on morphisms, and thus the 2-categorical point of view is forced upon us. We write down explicitly the fixed point map $\tilde{\Theta}^H\colon \cat(\tG, \sC)^H\ra \cat(\tG, \sD)^H$ induced from a pseudo equivariant functor $\Theta\colon \sC\ra \sD$, because it sheds  light on how the 2-cells come in, and because instances of this induced fixed point map are relevant in equivariant algebraic $K$-theory. We will encounter interesting maps in $K$-theory which turn out to be fixed point maps of equivariant $K$-theory maps that arise from pseudo equivariant functors on the categorical level.

 Recall the explicit description of homotopy fixed points given in \myref{homotopy fixed points}. 
Let $(C, f)$ be an object in $\cat(\tG, \sC)^H$.  Under the induced map on $H$-fixed points 
$$\tilde{\Theta}^H \colon \cat(\tG, \sC)^H \to \cat(\tG, \sD)^H, $$

\noindent this gets sent to $(\Theta(C), f_\theta)$ where $f_\theta(g)$ is defined as the composite $$\Theta(C)\xrightarrow[\cong]{\Theta(f(g))} \Theta(gC)\xrightarrow[\cong]{\theta_g} g\Theta(C).$$

Since $f(e)=\id$ and $\theta_e=\id$, it follows immediately that $f_\theta(e)=\id$. To show that $f_\theta$ satisfies the cocycle condition, we use the fact that $f$ satisfies it, together with the diagram in \myref{weakequivariant}. By that diagram, the maps in the following composite themselves factor as composites:

\[ \xymatrix{
\Theta(C)\ar[rrrr]^-{\Theta(f(gh))}\ar[rrd]_-{\Theta(f(g))} &&&&\Theta(ghC)  \ar[rrrr]^-{\theta_{gh}(C)} \ar[rdr]_-{\theta_g(hC)} &&&&gh\Theta(C)\\
& &\Theta(gC) \ar[rru]_-{\Theta(gf(h))}  & &&& g\Theta(hC) \ar[rru]_-{g\theta_h(C)} &&\\
}\]

\noindent We can use the naturality diagram for $\theta_g$

\[ \xymatrix{
\Theta(gC) \ar[rrr]^-{\Theta(gf(h))} \ar[dd]_-{\theta_g(C)} &&&\Theta(ghC)\ar[dd]^-{\theta_g(hC)}\\
&&&&&\\
g\Theta(C) \ar[rrr]_-{g\Theta(f(h))} &&& g\Theta(hC)
}\]

\noindent to replace the middle maps in the diagram above and we get that

$$\Theta(C)\xrightarrow{\Theta(f(gh))} \Theta(ghC)\xrightarrow{\theta_{gh}(C)} gh\Theta(C)$$ is the same as $$ \Theta(C) \xrightarrow{\Theta(f(g))} \Theta(gC) \xrightarrow{\theta_g(C)} g\Theta(C) \xrightarrow{g\Theta(f(h))} g\Theta(hC) \xrightarrow{g\theta_h(C)} gh\Theta(C).$$

\noindent Thus $f_\theta(gh)=(g\cdot f_\theta(h)) f_\theta(g)$.

\begin{question}
Does every equivariant functor $\cat(\tG, \sC)\ra \cat(\tG, \sD)$ come from a pseudo equivariant functor $\sC\ra \sD$?
\end{question}

\subsection{Homotopy invariance of homotopy fixed points revisited}
We can use pseudo equivariance to weaken the hypothesis of \myref{homotopy fixed point weak equiv} from requiring the functor to be on the nose equivariant to requiring it to be pseudo equivariant. Surprisingly, we get the same conclusion. 

\begin{proposition}\mylabel{pseudo weak equiv}
Let $\Theta \colon \sC \ra \sD$ be a pseudo equivariant functor which is a nonequivariant equivalence. Then the induced functor
$$\cat(\tG, \sC)\to \cat(\tG, \sD)$$ is a weak $G$-weak equivalence.
\end{proposition}

\begin{proof}\mylabel{homotopy fixed point weak equiv 2}
 The equivariant map $$\tilde{\Theta}\colon \cat(\tG, \sC)\ra \cat(\tG, \sD),$$  given in \myref{pseudo equiv} is a nonequivariant equivalence with inverse $\widetilde{\Theta^{-1}}$.

\noindent We have a commutative diagram: 
$$\xymatrix{
\cat(\tG, \sC) \ar[d]_\io \ar[rrr]^{\tilde{\Theta}} & &&  \cat(\tG, \sD) \ar[d]^\io \\
\cat(\tG, \cat(\tG, \sC)) \ar[rrr]^{\cat(\tG,-)(\tilde{\Theta})}  &&& \cat(\tG, \cat(\tG, \sD))
}$$
By \myref{homotopy fixed point weak equiv}, the bottom map is  a weak $G$-equivalence, and by  \myref{idem} the vertical maps are $G$-equivalences. Therefore the top map is a weak $G$-equivalence.

\end{proof}

\begin{cor}\mylabel{pseudo weak equiv cor}
A pseudo equivariant functor  $\Theta \colon \sC \ra \sD$ which is a nonequivariant equivalence induces equivalences of homotopy fixed points
$$\C^{hH}\ra \sD^{hH}$$ for all $H\subseteq G$.
\end{cor}

\section{Homotopy fixed points of module categories}

\subsection{$G$-rings and twisted group rings}
A \emph{$G$-ring} is a ring $R$ with a left action of $G$ by ring automorphisms. If $R$ is a topological ring, we ask for the action to be through continuous ring automorphisms. We have a homomorphism $G\rightarrow \Aut(R)$, and we write 
$g(r) = {}^gr$ for the automorphism $g\colon R\rtarr R$ determined by $g\in G$. 
Then ${}^{gh} r= g(h(r)) = {}^g({}^hr)$.

Note that when $R$ is a subquotient of $\bQ$, the only automorphism of $R$ is the identity 
and the action of $G$ must be trivial. However, it is well-known that even trivial $G$-actions on rings  yield nontrivial equivariant algebraic $K$-theory. For example, we will show that the equivariant algebraic $K$-theory of the topological rings $\R$ and $\bC$ with trivial $G$-action is equivariant topological real and complex $K$-theory. Nevertheless, we are interested in many examples where the group action on the ring is nontrivial such as the Galois extensions of rings, or
 the topological ring $\bC$ with $\Z/2\Z$   conjugation action.

Suppose that $R$ is a commutative $G$-ring with action given by $\theta\colon G\rtarr \Aut(R)$.  Observe that $R$ is an $R^G$-algebra, where $R^G$ is the subring of $G$-invariants. We can reinterpret $\theta$ as a group homomorphism $\theta\colon G\ra \End_{R^G} R$, and ask the question of when we can extend this to a ring map. More precisely, we seek to put a ring structure on the underlying abelian group of the group ring $R[G]$, for which  the map $\theta$ extends to a ring map. 

This naturally leads to the definition of  \emph{twisted group ring} (or \emph{skew group ring}), 
which we will denote by $R_G[G]$ (it is variously denoted in the literature also as $R\rtimes G$ or $R\ast G$). A more precise notation that takes into the action of $G$ on $R$ given by the homomorphism $\tha\colon G\rtarr \Aut(R)$ would be $R_\tha[G]$. However, the action of $G$ on $R$ will many times be implicit, so we will not adopt this more pendantic notation.

\begin{defn} As an $R$-module, \emph{the twisted, or skew, group ring} $R_G[G]$ is the same as the group ring $R[G]$, 
which is the case when $G$ acts trivially on $R$. 
We define the product on $R_G[G]$ by $R^G$-linear (not $R$-linear) extension 
of the relation
\[  (rg)\, (sh) = r\ \!{}^g\!s\, gh \]
for $r,s\in R$ and $g,h\in G$.  \end{defn}

Thus moving $g$ past $s$, ``twists" the ring element by the group action. Note that $R$ and $R^G[G]$ are subrings of $R_G[G]$ and 
\[  g\, r = {}^gr\, g. \]  Observe that the definition of the twisted multiplication in $R_G[G]$ is precisely what enables us to extend the group homomorphism $\theta\colon G\ra \End_{R^G} R$ to a ring homomorphism $$\theta\colon R_G[G] \ra \End_{R^G} R, \ \ \ \ \ \ \  (r\, g)\mapsto (s\mapsto r\ \!{}^g\!s).$$

\subsection{Modules over twisted group rings}

\begin{defn}
We call (left) $R_G[G]$-modules \emph{$G$-ring modules} or \emph{skew $G$-modules}.  
\end{defn}

\noindent Note that ${}^g(rs)=({}^gr)({}^gs)$ for all $r,s\in R$, thus $R$ is an example of an $R_G[G]$-module.

\begin{observation}\mylabel{semilinear}
An $R_G[G]$-module $M$ is a left $R$-module with a \emph{semilinear $G$-action}, i.e.,  $g(rm) = {}^gr(gm)$ for $m\in M$. If the action of $G$ on $R$ is trivial, then an $R[G]$-module is a left $R$-module $M$ with \emph{linear} $G$-action, namely, $g(rm)=r(gm).$ From this point of view an $R_G[G]$-linear map of $R_G[G]$-modules $f\colon M\ra N$ is a map of $R$-modules, which commutes with the $G$-action.
\end{observation}

If $G$ is finite and $|G|$ is invertible in $R$, we obtain the following characterization of projective modules over $R_G[G]$, which will be crucial in our applications to $K$-theory of $G$-rings.

\begin{proposition}\mylabel{invertible}
If $G$ is finite and $|G|^{-1}\in R$, then an $R_G[G]$-module is projective if and only if it is projective as an $R$-module.
\end{proposition}

\begin{proof}
An $R_G[G]$-module $M$ is projective if and only if the functor $$\Hom_{R_G[G]}(M,-): \Mod(R_G[G])\ra \Mod(R_G[G])$$ is exact. This functor is always left exact, and it is also right exact precisely when $M$ is projective. Let  $M$ and $N$ be $R_G[G]$-modules. As noted in \myref{semilinear},  $M$ and $N$  are $R$-modules with semilinear $G$-action. Then the $R_G[G]$-module $\Hom_{R_G[G]}(M,N)$ is the $R$-module $\Hom_R(M,N)$ with semilinear $G$-action given by conjugation, i.e., for an $R$-linear map $f\colon M\ra N$, $gf(m)=g(f(g^{-1}m))$. Again from  \myref{semilinear}, we have that $$\Hom_{R_G[G]}(M,N) \cong \Hom_R(M,N)^G.$$
The fixed point functor $(-)^G$ on $R_G[G]$-modules is right exact when the order of $G$ is invertible in $R$. Thus when $|G|^{-1}\in R$, the functor $\Hom_{R_G[G]}(M,-)$ is exact precisely when the functor $\Hom_R(M,-)$ is exact.

\end{proof}

Of course, we do not have a similar statement for free modules. Clearly, a free $R_G[G]$-module is free over $R$, but the converse is not true: Freeness over $R$ definitely does not imply  freeness over $R_G[G]$. For a set $A$, let $R[A]$ denote the free $R$-module on the basis $A$. The following proposition shows how we can put an $R_G[G]$-module structure on $R[A]$ if $A$ is a $G$-set; this is equivalent to specifying a semilinear $G$-action on $R[A]$.

\begin{prop}\mylabel{permrep}  Let $A$ be a $G$-set and define
\[  g ( \sum_a\,r_a a) =  \sum_a {}^gr_a ga \]
for $g\in G$, $r_a\in R$, and $a\in A$. Then $R[A]$ is an $R_G[G]$-module.
\end{prop}

In \cite[6.8]{GMM}, following \cite[5.1]{kawakubo} we gave a classification of $R_G[G]$-module structures on free rank $n$ $R$-modules in terms of the homotopy fixed point category of the group $GL_n(R)$, regarded as a single object groupoid. It inherits a $G$-action from the $G$-action on $R$.

\begin{thm}\mylabel{CrossR}  Let $R$ be a $G$-ring. Then the set of isomorphism
classes of $R_G[G]$-module structures on the $R$-module $R^n$ is in canonical bijective 
correspondence with the isomorphism classes of objects in the homotopy fixed point category $GL_n(R)^{hG}= \cat(\tG, GL_n(R))^G$.
\end{thm}

\subsection{The category $\cat(\tG, \Mod(R))$}\label{modules}

For a $G$-ring $R$, the category of finitely generated $R$-modules $\Mod(R)$ becomes a $G$-category with action defined in the following way. Let $M$ be an $R$-module with action map $\gamma\colon R\times M\rightarrow M$. Then we let $gM= M$ as abelian groups, and we define the action map by pulling back the action on $M$ along $g \colon R\rightarrow R$:  $$\gamma_g\colon R\times M\xrightarrow{g \times \id} R\times M\xrightarrow{\gamma} M.$$
This twists the $R$-action on $M$ by the action of $G$ on $R$. Explicitly, the $R$ action on $gM$, which we will denote by $\cdot_g$ to differentiate from the $R$-action on $M$, is given by $$r\cdot_g m :=  {}^grm,$$
where on the right hand side of the equation we are using the action of $R$ on $M$.\\

We note that $$R_G[G]\otimes_R M\cong \bigoplus_{g\in G} gM.$$

For a morphism $f\colon M\ra N$, we define $gf\colon gM \ra gN$ by $(gf)(m)=f(m)$. Thus $gf$ is the same as $f$ as a homomorphism of abelian groups, but it interacts differently from $f$ with the scalar multiplication.

Note that in general $M$ is not necessarily isomorphic to $gM$ as $R$-modules. The identity of abelian groups $M=gM$ is not an $R$-linear map, since the $R$-action is different on the two sides of the equality. However, we do have an isomorphism of free $R$-modules $R^n\cong gR^n$, which plays an important role.

\begin{lemma}
The $R$-modules $R^n$ and $gR^n$ are isomorphic.
\end{lemma}

\begin{proof}
Let $\{e_i\}$ be the standard basis for $R^n$. Note that this is also a basis for $gR^n$: if ${}^gr_1 e_1+\cdots +{}^g r_n e_n=0$, then ${}^g r_i=0$ for all $i$, so $r_i=0$ since $G$ acts by ring automorphisms. Also, every element in $gR^n$ can be written as 
$$(r_1,\cdots, r_n)= {}^g({}^{(g^{-1})} r_1) \ e_1+\cdots +{}^g({}^{(g^{-1})}r_n)\  e_n.$$ Now just define a map on basis elements as the identity
$e_i \mapsto e_i$ and extend linearly, i.e. $re_i\mapsto {}^g r e_i$.
\end{proof}

We emphasize that the objects of the category $\Mod(R)$ are $R$-modules $M$, which know nothing about the $G$-action on $R$. We used this action  to define a $G$-action on the category $\Mod(R)$, and now we will show how  the $G$-category $\Mod(R)$ relates  to the category of modules over the twisted group ring $R_G[G]$, which by \myref{semilinear} is the same as the category of $R$-modules with semilinear $G$-action.

\begin{proposition}\mylabel{equiv chain}
The homotopy fixed point category $\Mod(R)^{hG}$ is equivalent to the category $\Mod(R_G[G])$.
\end{proposition}

\begin{proof}
From the description of homotopy fixed point categories given in \myref{wrong homotopy fixed points}, the objects of the homotopy fixed point category $\cat(\tG, \Mod(R))^G $ are $R$-modules $M$ together with compatible isomorphisms $f(g)\colon gM\xra{\cong} M$, one for each element $g\in G$, for which $f(e)=\id_M$ and which make the diagrams
$$\xymatrix{
(gh)M\ar[dr]_{f(gh)} \ar[rr]^{gf(h)} && gM \ar[dl]^{f(g)}\\
 & M &
}$$
commute.

Define an action of $G$ on $M$ by $g\cdot m=f(g)(m)$. This is indeed an action since $f(e)=\id_M$ and 
\begin{eqnarray*}
(gh)\cdot m &=& f(gh)(m)\\ 
&=& f(g)gf(h)(m)\\
&=& f(g)f(h)(m)\\
&=& g\cdot(h\cdot m).
\end{eqnarray*}

\noindent The second to last identification is just the definition of the $G$-action on morphisms of modules in $\Mod(R)$; the morphism $gf(h)$ is the same as $f(h)$ as a morphism of abelian groups.

Now note that this action is indeed semilinear:
$$g\cdot(r m)=f(g)(r\cdot_g m)= {}^g r f(g)(m)= {}^g r (g\cdot m).$$

\noindent Via this identification, the morphisms in the homotopy fixed point category are precisely the $G$-equivariant maps of $G$-modules.

Thus we have shown that the homotopy fixed point category $ \Mod(R){^hG} $ can be identified with the category of modules with semilinear $G$-action. Combining this with \myref{semilinear}, we obtain the desired result.
\end{proof}

By  \myref{hfixedpoints},  we immediately get the following corollary.

\begin{cor}
The homotopy fixed point category $\Mod(R)^{hH}$ is equivalent to the category $\Mod(R_H[H])$ for all subgroups $H\subseteq G$.
\end{cor}

Therefore, the $G$-category $\cat(\tG, \Mod(R))$ encodes the module categories over the twisted group rings for all subgroups $H$ as fixed point subcategories. Thus by studying the equivariant object $\cat(\tG, \Mod(R))$ we are implicitly studying the representation theory of all the subgroups at once.

Let $\sP(R)$ be the category of finitely generated projective $R$-modules. This becomes a $G$-category in the same way that $\Mod(R)$ does since $gP$ is projective if $P$ is so: if $P\oplus Q\cong R^n$, then $gP\oplus gQ\cong gR^n \cong R^n$.   The proof of \myref{equiv chain} goes through to show that the category $\cat(\tG, \sP(R))$ is equivalent to the category of finitely generated projective $R$-modules with semilinear $G$-action. Therefore, by \myref{invertible}, if $G$ is finite and the order of $G$ is invertible in $R$, we obtain \myref{equiv chain} and its corollary if we restrict to the category of finitely generated $R$-modules.

\begin{proposition}\mylabel{equiv chain proj}
Suppose $G$ is finite and $|G|^{-1}\in R$. The homotopy fixed point category $\sP(R)^{hG}$ is equivalent to the category $\sP(R_G[G])$.
\end{proposition}

\begin{cor}\mylabel{equiv chain proj cor}
Suppose $G$ is finite and $|G|^{-1}\in R$. The homotopy fixed point category $\sP(R)^{hH}$ is equivalent to the category $\sP(R_H[H])$ for all subgroups $H\subseteq G$.
\end{cor}

\subsection{The equivariant skeleton of free modules}\mylabel{skeleton}

If $M\cong R^n$, then $gM\cong gR^n\cong R^n$, so the $G$-action on $\Mod(R)$ restricts to an action on the category $\sF(R)$ of finitely generated free $R$-modules.

\begin{definition}\mylabel{glr}
 Let $\mathscr{GL}(R)$ be the category with objects the based free $R$-modules $R^n$ and morphism spaces
$$ \Mor_{\sGL(R)} (R^n, R^m)= \begin{cases} \emptyset &\mbox{if } n \neq m \\ 
GL_n(R) & \mbox{if } n =m. \end{cases} $$
\end{definition}
\noindent This is the same as the disjoint union of the one object categories $GL_n(R)$, i.e., $$\mathscr{GL}(R)= \coprod_{n\geq 0} GL_n(R),$$ and it is a skeleton of the category of $\iso \sF(R)$ of finitely generated free $R$-modules and isomorphisms.

 We note that in general, even if $\sC$ is a $G$-category, the skeleton $\sk\C$ is not closed under the $G$-action\footnote{We treat this general case in the appendix.}. However, if $R$ is a $G$-ring, we have an obvious action on $\sGL(R)$: it is trivial on objects and on morphisms $g$ acts entrywise. Clearly, the inclusion of the skeleton
 $$i\colon \sGL(R)\to \iso\sF(R)$$ is not an equivariant map since the object $R^n$ is fixed in $\sGL(R)$ but not in $\iso\sF(R)$. However, we can define an inverse to it which is equivariant. Fix isomorphisms $\ga_M\colon M\xra{\cong} R^k$ for all finitely generated free modules $M$, i.e., fix a basis $\{\ga_M^{-1}(e_i)=m_i\} $ for all $M$ such that $\ga_M=\ga_{gM}$ as isomorphisms of abelian groups. In other words, we pick the same basis for $M$ and $gM$; recall that $M$ and $gM$ are equal as abelian groups. We define the inverse equivalence $i^{-1}$ by $M\mapsto R^k$ on objects. Given an isomorphism  $M\ra N$ in $\iso\sF(R)$,  it maps to the composite $$R^k\xra{\ga_M^{-1}} M \xra{f} N \xra{\ga_N} R^k.$$
 
 We show that the map $i^{-1}$ is equivariant. Clearly, it  commutes with the $G$-action on objects, since the action is trivial in $\sGL(R)$ and if $M$ has dimension $k$ so does $gM$. Now let $f\colon M\ra N$ be an isomorphism in $\iso \sF(R)$, and suppose that $$f(m_i)= r_{i1}n_1 +\cdots + r_{ik}n_k.$$ 
 
\noindent  The morphism $gM\xra{gf} gN$ maps to $$R^k\xra{\ga_M^{-1}} gM \xra{gf} gN \xra{\ga_N} R^k.$$ On basis elements, this is $$(gf)(e_i)= r_{i1} \cdot_g e_1+\cdots + r_{ik}\cdot_g e_k = r_{i1}^g e_1 +\cdots + r_{ik}^g.$$ 
 
 \noindent Therefore, we get entrywise action by $g$ on the matrix representing $f$, and the map $$i^{-1} \colon \iso\sF(R) \to \sGL(R)$$ is $G$-equivariant. It is a nonequivariant equivalence, thus by \myref{homotopy fixed point weak equiv} we obtain the following result.
 
 \begin{proposition}\mylabel{skeleton equiv}
 Suppose $R$ is a $G$-ring. Then there is a weak $G$-equivalence $$\cat(\tG, \iso\sF(R))\to \cat(\tG, \sGL(R)).$$  
 \end{proposition}
 
 This is very useful because it will allow us to use the skeleton $\sGL(R)$ in equivariant algebraic $K$-theory without losing information about the entire category of free modules with its induced action of $G$. 

\subsection{Equivariant skeleta}\label{skeleta}
Nonequivariantly, it is always assumed in $K$-theory that when we take the classifying space of a category which is not small, such as $\sP(R)$, $\sF(R)$, or $\Mod(R)$, we are tacitly replacing the category by a small category which is equivalent to it, such as its skeleton.

As we have seen in Section \ref{skeleton}, the situation is a little trickier equivariantly, because we do not have an equivariant equivalence between a $G$-category and its skeleton. This is too much to hope for; however, we show that the discussion in Section \ref{skeleton} 
generalizes. What makes the general case trickier is the fact that unlike in the case of free modules where we showed that $R^n\cong gR^n$, in general, an object $C$ is not necessarily isomorphic to $gC$. 

We show that for a $G$-category $\C$, we can put a $G$-action on the skeletal category $\sk \C$, such that the inverse of the inclusion of the skeleton $i\colon \sk\C \ra \C$ is a $G$-map which is a nonequivariant equivalence. This implies by \myref{homotopy fixed point weak equiv} that the map $$\cat(\tG, \C)\to \cat(\tG, \sk \C)$$ is a weak $G$-equivalence. This suffices for our applications, because in equivariant algebraic $K$-theory we are only taking classifying spaces of categories of the form $\cat(\tG, \C)$.

For an object $C\in \C$, denote by $C^{rep}$ the representative of the isomorphism class of $C$ in $\sk \C$, so that if $C\cong D$, then $C^{rep}=D^{rep}$. We fix isomorphisms $\ga_C\colon C\xra{\cong} C^{rep}.$ The map
$$i^{-1}\colon \C\to \sk\C$$ is defined on objects as $C\mapsto C^{rep}$ and on morphisms as $$(C\xra{f} D)\  \mapsto \ (C^{rep}\xra{\ga_C^{-1}} C\xra{f} D \xra{\ga_D} D^{rep}).$$  

We define a $G$-action on $\sk\C$ in the following way. On objects, $$gC^{rep}:=(gC)^{rep}.$$
We remark that there is no way to consistently pick the representatives such that $gC^{rep}=(gC)^{rep}$ is an equality in $\C$. However, we do have isomorphisms in $\C$
$$(gC)^{rep} \xra[\ga_{gC}^{-1}]{\cong} gC \xra[g\ga_C]{\cong} gC^{rep}.$$

We define the action on morphisms of $\sk\C$. We defined $g(C^{rep}\xra{f} D^{rep})$ as
$$(gC)^{rep}\xra[\ga_{gC}^{-1}]{\cong} gC \xra[g\ga_C]{\cong} gC^{rep} \xra{gf} gD^{rep} \xra[g\ga_D]{\cong} gD \xra[\ga_{gD}]{\cong} (gD)^{rep}.$$

Now the map $i^{-1}$ is clearly equivariant on objects. We show it is also equivariant on morphisms. Let $f\colon C\ra D$ be a morphism in $\C$, which gets mapped by $i^{-1}$ to $C^{rep}\xra{\ga_C^{-1}} C\xra{f} D \xra{\ga_D} D^{rep}$ in $\sk\C$. Acting by $g$, we get $$(gC)^{rep}\xra[\ga_{gC}^{-1}]{\cong} gC \xra[g\ga_C]{\cong} gC^{rep}\xra[g\ga_C^{-1}]{\cong} gC \xra{gf} gD \xra[g\ga_D]{\cong} gD^{rep} \xra[g\ga_D]{\cong} gD \xra[\ga_{gD}]{\cong} (gD)^{rep}.$$
By composing the inverse isomorphism, this is the same as 
$$(gC)^{rep}\xra[\ga_{gC}^{-1}]{\cong} gC \xra{gf} gD \xra[\ga_{gD}]{\cong} (gD)^{rep},$$ 
which is just $i^{-1}$ applied to $gC \xra{gf} gD$, and therefore, the map $i^{-1}\colon \C\ra \sk\C $ is a $G$-map for the action we defined on $\sk\C$.

Since the equivariant algebraic $K$-theory construction involves replacing the usually non-small $G$-category of interest $\sC$ with $\cat(\tG, \sC)$, we can with clear conscience assume use of equivariantly skeletally small models when we apply the classifying space functor $B$. 

 \subsection{Equivariant Morita theory}\label{morita}
 
 We give a definition of \emph{equivariant Morita equivalence} of $G$-rings; the philosophy is that  this notion should capture Morita equivalences of  twisted group rings.
 
 \begin{definition}\mylabel{morita def}
 Two $G$-rings $R$ and $S$ are \emph{equivariantly Morita equivalent} if they are nonequivariantly Morita equivalent and the equivalence $$\Mod(R)\ra \Mod(S)$$ is pseudo equivariant.
 \end{definition} 
 
 In \cite{biland}, Biland gives a definition of equivariant Morita equivalence, and it is easy  to see that his definition agrees with ours\footnote{The only difference in the definitions is that Biland does not require $\theta_e=\id$, but we suspect that is a typo in his preprint. Also, we note that his notion of equivariance is not the standard one; it corresponds to  our notion of pseudo equivariance.}.
 Biland shows that \myref{morita def} is equivalent to having a $G$-equivariant bimodule, which provides the equivariant Morita equivalence. For the definition of $G$-bimodule and the details of the equivalence of the two statements we refer the reader to Biland's preprint \cite[Thm. A]{biland}.

 Note that a consequence of our definition of equivariant Morita equivalence and \myref{pseudo weak equiv} is the following proposition.
 
 \begin{proposition} If two $G$-rings $R$ and $S$ are equivariantly Morita equivalent, then there is an equivariant weak equivalence $$\cat(\tG, \Mod(R))\to \cat(\tG, \Mod(S)).$$ \end{proposition}
 Thus we have a $G$-map which induces an equivalence on all fixed points
$$\cat(\tG, \Mod(R))^H\to \cat(\tG, \Mod(S))^H.$$ As we have shown in \myref{equiv chain} this ensures that the twisted group rings $R_H[H]$ and $S_H[H]$ are Morita equivalent in the classical sense for all $H\subseteq G$ .

We end with a consequence of equivariant Morita equivalence, which will be relevant in algebraic $K$-theory. Recall that a nonequivariant Morita equivalence $\Mod(R)\ra \Mod(S)$ restricts to an equivalence $\sP(R)\ra \sP(S)$ on the categories of finitely generated projective modules (for example, see \cite[II, 2.7.]{weibel}). 

\begin{lemma}\mylabel{morita proj}
If $R$ and $S$ are equivariantly Morita equivalent, then there is a weak  $G$-equivalence $$\cat(\tG, \sP(R)) \ra \cat(\tG, \sP(S))$$ which induces equivalences of the homotopy fixed point categories of finitely generated projective modules $\sP(R)^{hH} \ra \sP(S)^{hH}$ for all $H\subseteq G$.
\end{lemma}

\begin{proof}
Since $R$ and $S$ are equivariantly Morita equivalent, by definition we have a nonequivariant Morita equivalence $\Mod(R)\ra \Mod(S)$, which is pseudo equivariant. This restricts to an equivalence $\sP(R)\ra \sP(S)$, which is pseudo equivariant, and we get the result by applying \myref{pseudo weak equiv}. The second statement follows by passing to fixed points.
\end{proof}

\section{Equivariant algebraic $K$-theory of $G$-rings}
 Nonequivariantly, the algebraic $K$-theory space of $R$ is defined as the group completion of the classifying space of the symmetric monoidal category $\iso \sP(R)$ of finitely generated projective modules and isomorphisms, and this space is delooped using an infinite loop space machine such as  the operadic one developed by May in \cite{MayGeo} or the one based on $\Gamma$-spaces developed by Segal in \cite{segal}. 

The category $\iso \sP(R)$ is a $G$-category with action defined as in the previous section, and it is not hard to see that it yields a \emph{naive} $\OM$-$G$-spectrum, i.e., an $\OM$-spectrum with $G$-action. The Segalic infinite loop space machine has been generalized equivariantly by Shimakawa in \cite{Shimakawa}, and the operadic infinite loop space machine has been generalized equivariantly by Guillou and May in \cite{GM3}, to give \emph{genuine} $\OM$-$G$-spectra. We show in \cite{MMO} that the equivariant generalizations yield equivalent genuine $G$-spectra when fed equivalent input, so we can use either machine to deloop equivariant algebraic $K$-theory. We  describe these machines and the input they take in section \ref{delooping section}. It turns out that a symmetric monoidal category with $G$-action such as $\iso \sP(R)$ is inadequate input for these machines, but $\cat(\tG, \iso \sP(R))$ is the ``genuine" kind of input these machines take to produce a genuine $\OM$-$G$-spectrum with zeroth space the equivariant group completion of $B\cat(\tG, \iso \sP(R))$. 

The equivariant infinite loop space machines will provide a functorial model for the equivariant group completion of $B\cat(\tG, \iso \sP(R))$; however, we will first define the equivariant algebraic $K$-theory space of a $G$-ring $R$ via an explicit model of the equivariant group completion of the classifying space of a symmetric monoidal $G$-category, namely the equivariant version of Quillen's $S^{-1}S$-construction. This model allows us to run an equivariant version of the first part of Quillen's ``plus=Q" proof relating the group completion of the symmetric monoidal category of finitely generated projective $R$-modules and isomorphisms to the group completion of the topological monoid $\coprod BGL_n(R)$ of classifying spaces of principal $GL_n(R)$-bundles. Using this model, we show that the equivariant algebraic $K$-theory space of a $G$-ring $R$ is equivalent on higher homotopy groups with an equivariant interpretation of the ``plus" construction, namely the group completion of a topological $G$-monoid of $G$-equivariant $GL_n(R)$-bundles.

\subsection{Symmetric monoidal $G$-categories} 
We define a \emph{symmetric monoidal $G$-category} as a strict symmetric monoidal category $\C$ with $G$-action which commutes with the symmetric monoidal structure. Concisely, this is a functor $G\ra \mathrm{Sym} \cat^\mathrm{strict}$, from $G$ to the category of symmetric monoidal categories and \emph{strict} monoidal functors. However,  in practice, some symmetric monoidal categories that we care about have a $G$-action which preserves the monoidal structure only up to isomorphism, i.e., they are functors $G\ra \mathrm{Sym} \cat^\mathrm{strong}$, from $G$ to the category of strict symmetric monoidal categories and \emph{strong} monoidal functors. 

 The problem with the latter  is that if the $G$-action preserves the monoidal structure only up to isomorphism, the fixed point subcategories $\C^H$ are not necessarily closed under the monoidal structure. We show that applying the functor $\cat(\tG, -)$ rectifies symmetric monoidal categories for which the action functors $g\cdot$ are only strong symmetric monoidal to symmetric monoidal $G$-categories with action that preserves the monoidal structure on the nose. Therefore the homotopy fixed point categories $\C^{hH}$ are closed under the symmetric monoidal structure.\footnote{This issue becomes even more subtle for Waldhausen categories, which is analyzed in forthcoming work with C. Malkiewich.} 

Suppose that $\C$ is defined by a functor $G\ra \mathrm{Sym} \cat^\mathrm{strong}$. Then  $\C$ is a symmetric monoidal category for which the symmetric monoidal structure map $$\C\times \C \xra{\oplus} \C$$ is pseudoequivariant, where the $G$-action on $\C\times \C$ is diagonal, and for which  $gI\cong I$ for every $g\in G$, where $I$ is the unit object of $\C$. By \myref{pseudo equiv}, we get an on the nose equivariant functor

$$\oplus\colon \cat(\tG, \C\times \C) \cong \cat(\tG, \C) \times \cat(\tG, \C) \ra \cat(\tG, \C).$$ For $F_1, F_2$, the functor $F_1\oplus F_2$ is defined on objects as

$$(F_1\oplus F_2)(g)= g\big(g^{-1}F_1(g)\oplus g^{-1}F_2(g)\big),$$ which, of course, is the same as $F_1(g)\oplus F_2(g)$ when the $G$-action on $\C$ preserves $\oplus$ strictly. On a morphism $(g',g)$, $F_1\oplus F_2$ is defined as 
$$(F_1\oplus F_2) (g)\xra{\cong} F_1(g)\oplus F_2(g) \xra{F_1(g',g)\oplus F_2(g',g)}    F_1(g')\oplus F_2(g')\xra{\cong} (F_1\oplus F_2)(g').$$

It is not hard to see that the functor $F_I\colon \tG \ra \C$ defined by $F_I(g)=gI$, where $I$ is the unit of $\C$ is a unit for the symmetric monoidal structure defined above. Therefore, even when the $G$-action on $\C$ does not preserve the symmetric monoidal structure strictly, $\cat(\tG, \C)$ does become a symmetric monoidal $G$-category for which the action commutes with the symmetric monoidal structure.

\subsection{The equivariant group completion of the classifying space of a symmetric monoidal $G$-category}

 A Hopf $G$-space is a Hopf space with equivariant multiplication map and for which multiplying by the identity element is $G$-homotopic to the identity map such as, for example, $\OM X$ for a $G$-space $X$. The equivariant notion of group completion is captured by the fixed point maps being group completions.
	
	\begin{definition}
	A $G$-map $X\ra Y$ of homotopy associative and commutative Hopf $G$-spaces is an \emph{equivariant group completion} if the fixed point maps $X^H\ra Y^H$ are group completions for all $H\subseteq G$.
	\end{definition}
	
\begin{rmk}\mylabel{gpcompl} For a homotopy commutative topological $G$-monoid,  since the classifying space functor $B$ and the loop functor $\OM$ have the wonderful virtue of commuting with fixed points, the nonequivariant group completion map $M\ra \OM BM$ (\cite{mcduff}, \cite{MayClass}) is an equivariant group completion.\end{rmk}

	Note that the classifying space $B\C$ of a symmetric monoidal $G$-category is a Hopf $G$-space. We give a functorial construction of the group completion of $B\C$, following \cite{quillen1}.  The idea is to define the group completion on the category level. 	 We recall  the model for the categorical group completion in the nonequivariant case, and then we observe that the theory carries through equivariantly \emph{as long as the $G$-action preserves the symmetric monodical structure strictly.} 

\begin{definition}[\cite{quillen1}] Let $S$ be a symmetric monoidal category. The category $S^{-1}S$ has objects pairs $(m, n)$ of objects in $S$. A morphism $(m, n) \ra (p, q)$ in $S^{-1}S$ is an equivalence class of triples
$$(r, \ r\oplus m\xra{f} p, \ r\oplus n\xra{g} q)$$ where two triple are equivalent if there is an isomorphism of the  first entries that makes the relevant diagrams commute.
Composition for a pair of morphisms is defined as 
$$(r,  r\oplus m\xra{f} p, r\oplus n\xra{g} q) \circ (s,  s\oplus p\xra{\phi} u, s\oplus q\xra{\psi} v)=(s\oplus r,s\oplus r\oplus m\xra{\phi\circ(s\oplus f)} u,  s\oplus r\oplus n \xra{\psi \circ (s\oplus g)} v).$$
\end{definition}  

Note that $S^{-1}S$ is symmetric monoidal with $(m,n)\oplus (p,q)= (m\oplus p, n\oplus q)$, and there is a strict monoidal functor $S\ra S^{-1}S$ given by $m\mapsto (m,0)$, where $0$ is the unit of $S$. This induces a map of Hopf spaces \begin{equation}\label{bs map} BS \to BS^{-1} S, \end{equation} 

\noindent which, subject to a mild condition, was shown by Quillen to be a group completion when $S$ is a groupoid.

\begin{theorem}[\cite{quillen1}]\mylabel{quillen bs}
Let $S$ be a symmetric monoidal groupoid such that translations are faithful. i.e., $$\Aut(s)\ra \Aut(s\oplus t)$$ is injective for all $s,t\in S$. Then the map $BS\ra BS^{-1}S$ is a group completion.
\end{theorem}

Now if $S$ is a symmetric monoidal  $G$-category with $G$-action that preserves $\oplus$, then $S^{-1}S$ is also a symmetric monoidal  $G$-category with diagonal action on objects. On morphisms, 
$$g\big((m,n) \xra{(r, \ f,\ f')} (p, q)\big)= (gm, gn)\xra{(gr, \ gf,\  gf')} (gp, gq).$$
Note that this only works, because the action of $G$ commutes with $\oplus$. The fixed point subcategory $S^H$ is also a symmetric monoidal category, thus we can form $(S^H)^{-1}(S^H)$, and it is not hard to see that the construction commutes with fixed points.

\begin{lemma}
Let $S$ be a symmetric monoidal $G$-category. Then 
$$(S^H)^{-1}(S^H) \cong (S^{-1}S)^H$$ for all $H\subseteq G$.
\end{lemma}

Also, note that if translations are faithful in $S$, i.e., if $\Aut(s)\ra \Aut(s\oplus t)$ is injective for all $s,t\in S$, then the same holds for the fixed point subcategories $S^H$. This has the following immediate consequence.

\begin{proposition}\mylabel{bs group completion}
Let $S$ be a symmetric monoidal $G$-groupoid such that translations are faithful.
Then the map $BS\ra BS^{-1}S$ is an equivariant group completion.
\end{proposition}

Note that restricting to the subcategory of isomorphisms commutes with fixed points, namely $(\iso\C)^H=\iso(\C^H)$ for subgroups $H\subseteq G$. Since the algebraic $K$-theory of symmetric monoidal categories, and in particular the $K$-theory of rings, is defined with respect to the class of isomorphisms, this observation is crucial. We will also need the following useful result stating that restricting to the subcategory of isomorphisms commutes with applying the $\cat(\tG, -)$ functor.

\begin{lemma}\mylabel{iso}
For any $G$-category $\C$, we have an identification  $$\iso \cat(\tG, \C)\cong \cat(\tG, \iso \C).$$
\end{lemma}

\begin{proof}
Note that a functor $F\colon \tG \ra \C$ actually lands in $F\colon \tG \ra \iso\C$ since every morphism in $\tG$ is an isomorphism. Therefore the objects of $\iso \cat(\tG, \C)$ and $\cat(\tG, \iso \C)$ are the same. Now a morphism in $\iso \cat(\tG, \C)$ is a natural transformation whose component maps are all isomorphisms, which is the same with a morphism in $\cat(\tG, \iso\C).$
\end{proof} 
 
 Now, using the explicit model for group completion we make the following definition of the equivariant algebraic $K$-theory space of a $G$-ring $R$. 

\begin{definition}\mylabel{bs def}
The equivariant algebraic $K$-theory space of a  $G$-ring $R$ is the $G$-space $K_G(R)= B(S^{-1}S),$ where $S$ is the symmetric monoidal $G$-category $\cat(\tG, \iso \sP(R))$.
\end{definition}

We define the equivariant $K$-groups as the equivariant homotopy groups of this space. Once we deloop this space equivariantly to a genuine  $\OM$-$G$-spectrum, it will turn out that these are the homotopy groups of that spectrum, and therefore, they have Mackey functor structure. Recall that for a subgroup $H\subseteq G$ and a $G$-space $X$, we have  
\begin{equation}\mylabel{pih}
\pi_i^H(X)=[(G/H)_+ \wedge S^i, X]_G \cong [S^i, X^H]=\pi_i(X^H),
\end{equation}
 where $[X,Y]_G$ denotes the set of homotopy classes of based $G$-maps $X\rightarrow Y$ between based $G$-spaces, and $X_+$ denotes the union of $X$ with a disjoint basepoint. We define the $K$-groups for $i\geq 0$.
 
\begin{definition}\mylabel{higher groups}
The algebraic $K$-theory groups are given by
$$K_i^H (R) = \pi_i^H(K_G(R)).$$
\end{definition}

\begin{remark}

We spell out what the equivariant $K_0$ is so that it is clear how it relates back to the nonequivariant algebraic definition. Nonequivariantly, $K_0(R)$ is the group completion of the abelian monoid $\mathcal{P}(R)$ of finitely generated projective $R$-modules, i.e., it is $\pi_0$ of the topological group completion of  $B(\iso\sP(R))$, the classifying space of the category of finitely generated projective modules and isomorpshisms. Equivariantly,  $K_0^H(R)\cong \pi_0(K_G(R)^H)$, and
$K_G(R)^H$ is the group completion of $B\cat(\tG, \iso\sP(R))^H$. Therefore $K_0^H(R)$ is the group completion of the abelian monoid of isomorphism classes of objects in the homotopy fixed point category $\iso \sP(R)^{hH}=\iso \sP(R_H(H))$, and thus it agrees with $K_0$ of the twisted group ring $R_H(H)$.
\end{remark}

 \subsection{``Plus" construction interpretation and connection to equivariant bundle theory}\mylabel{pluss}
 Quillen's first definition of higher algebraic $K$-groups was as the homotopy groups of the space $BGL(R)^+$, which he showed to be homotopy equivalent to the basepoint component of the group completion of the  topological monoid $B(\coprod_n GL_n(R))\cong\coprod_n BGL_n(R)$. Note that  this is the  monoid of classifying spaces of principal $GL_n(R)$-bundles under Whitney sum.  Fiedorowicz, Hauschild and May gave a first definition of equivariant algebraic $K$-groups of a ring $R$ with \emph{trivial} $G$-action  in \cite{FHM} by replacing this space with the monoid of classifying spaces of equivariant bundles. However, in their definition, since $G$ does not act on $R$, the equivariance group $G$ does not act on the structure group $GL_n(R)$ of the bundles; they are considering equivariant bundles that have commuting actions of $G$ and $GL_n(R)$ on the total space.
 
We generalize the definition of Fiedorowicz, Hauschild and May so as to allow nontrivial action of $G$ on the ring $R$. Instead of using the classifying spaces of equivariant $(G,GL_n(R))$-bundles, which correspond to a trivial group extension $$1\rightarrow GL_n(R) \rightarrow GL_n(R)\times G \xrightarrow{q} G \rightarrow  1,$$ we will use the classifying spaces of $(G, GL_n(R)\rtimes G)$-bundles, which correspond to split extensions
	$$1\rightarrow GL_n(R) \rightarrow GL_n(R) \rtimes G \xrightarrow{q} G \rightarrow 1.$$ 
For a precise definition of such equivariant bundles, see \cite{GMM}, or any of the earlier sources cited there. Suitable categorical models for classifying space of  $(G, GL_n(R)\rtimes G)$-bundles have been constructed in \cite{GMM}, and these are central to our definition. The relevant theorem is the following.
 
 \begin{theorem}[\cite{GMM}]\mylabel{bundle thm}
 If $G$ is discrete and $\PI$ is discrete or compact Lie, then $B\cat(\tG, \PI)$ is a classifying space for $(G, \PI \rtimes G)$-bundles. 
 \end{theorem}
 
 Recall from \myref{glr} that $\sGL(R)= \coprod_n GL_n(R)$. By \myref{bundle thm}, the monoid of classifying spaces of $(G, GL_n(R)_G)$-bundles is $B\cat(\tG, \sGL(R))$, and a model for the group completion is, by \myref{gpcompl}, $$\OM BB\cat(\tG, \sGL(R)).$$

We proceed to show that we have an equivalence $$K_G(R)^H\simeq \OM BB\cat(\tG, \sGL(R))^H$$ on basepoint components, so these spaces have the same higher homotopy groups. This shows that the definition of \cite{FHM} of higher equivariant $K$-groups for a ring with trivial $G$-action agrees with the one given in \myref{higher groups}.

  Again, we will follow Quillen's nonequivariant proof \cite{quillen1}. We recall the definition of cofinality and then state the result that leads to showing that cofinality gives an equivalence on higher $K$-theory.  

\begin{definition}
A monoidal functor $F \colon S \ra T$ is cofinal if for every $ t\in T$ there is $t'\in T$ and $s \in S$ such that $ t\oplus t' =F (s).$
\end{definition}
\begin{proposition}[\cite{quillen1}]\mylabel{cofinality}
If $F \colon S \ra T$ is cofinal and $\Aut_{S}(s) \cong \Aut_{T} (F(s))$ for all $s\in S$, then the map $B(S^{-1}S) \ra B(T^{-1}T )$ induces an equivalence of  basepoint components.
\end{proposition}

Nonequivariantly this is applied to the cofinal inclusion $\iso \sF(R)\hookrightarrow \iso \sP(R)$, which is in fact the idempotent completion. We wish to apply it to the fixed point maps of the inclusions
$$\cat(\tG, \iso \sF(R))^H\to \cat(\tG, \iso \sP(R))^H.$$

For this we recall from section \ref{modules} that the homotopy fixed point category $\cat(\tG, \iso \sF(R))^H$ is the category of finitely generated free $R$-modules with semilinear $H$-action and $H$-isomorphism. Similarly, $\cat(\tG, \iso \sP(R))^H$ is the category of finitely generated projective $R$-modules with semilinear $H$-action. This inclusion is cofinal; however it is not an idempotent completion anymore. One way to see this is to consider the case when $G$ is finite and $|G|^{-1}\in R$, so that, by \myref{equiv chain proj cor}, $\cat(\tG, \iso \sP(R))^H\simeq \sP(R_H[H])$. We have a commutative diagram of inclusions:

$$\xymatrix{
\cat(\tG, \iso \sF(R))^H \ar[rrr]  &&& \sP(R_H[H])\\
\sF(R_H[H]) \ar[u] \ar[rrru] &&&&
}$$

The diagonal map is an idempotent completion, but the map going straight up is not an equivalence, since free $R_H[H]$-modules do not coincide with modules with semilinear $G$-action which are free as $R$-modules.  Therefore, the top map is not the idempotent completion; it just factors it.

\begin{thm}\mylabel{plus equals ss}
There is an equivalence on connected basepoint components $$\OM BB\cat(\tG, \sGL(R))^H \simeq_0 K_G(R)^H.$$ \end{thm}

\noindent (We used the notation $\simeq_0$ instead of $\simeq$ in order to emphasize that this equivalence only holds on basepoint components and to avoid possible misinterpretation.)
 \begin{proof}
 
 By \myref{skeleton equiv},  the inverse of the nonequivariant equivalence given by the inclusion of the skeleton $i\colon \sGL(R)\ra \iso\sF(R)$ induces a weak $G$-equivalence $$\cat(\tG, \iso\sF(R))\ra \cat(\tG, \sGL(R)).$$
 
Now $$\cat(\tG, \iso \sF(R))^H\to \cat(\tG, \iso \sP(R))^H$$ is cofinal. Therefore, by applying \myref{cofinality}, we get a weak $G$-equivalence of basepoint components $$B(S^{-1}S) \ra B(T^{-1}T )$$ for $S=\cat(\tG, \iso \sF(R))^H$ and $T=\cat(\tG, \iso \sP(R))^H$. 
 \end{proof}

\begin{remark}
We note that equivariantly there is no meaningful way to write down a decomposition of the $K$-theory space as a product of $K_0$ and a connected component, analogous to the widely used nonequivariant one, which is $$K(R)\simeq K_0(R)\times \OM_0BB\sGL(R).$$

The reason is that taking basepoint components does not commute with taking fixed points, so if we split off the basepoint component we change the equivariant homotopy type of the space.
  However, even nonequivariantly, this decomposition is not functorial, so it is technically more correct to define the $K$-theory space via a functorial model for the group completion of the classifying space of the symmetric monoidal category of finitely generated projective modules and isomorphisms.

  \end{remark}

\subsection{Equivariant delooping of the $K$-theory space}\label{delooping section}

We describe the May and Segal equivariant infinite loop space machines and the input they take. By a celebrated theorem of May and Thomason, the nonequivariant infinite loop space machines are equivalent. Their proof does not generalize equivariantly, but we have shown  in \cite{MMO} through a surprising chain of equivalences that the equivariant generalizations of the machines also produce equivalent $\OM$-$G$-spectra. So, up to equivalence,  we could use either machine to define the equivariant algebraic $K$-theory spectrum of a $G$-ring. The construction of each machine has its own advantages, and in some applications we have in mind we will need the specifics of one machine over the other. However, for the rest of this paper, we will study the equivariant homotopy type of the spectrum we get, and we will not need the specifics of either of these machines. We choose to define equivariant algebraic $K$-theory of a $G$-ring using the equivariant May machine, but we describe the alternative construction using the equivariant Segal machine and invoke the theorem by which they are equivalent. 
 
 \subsubsection{Equivariant May infinite loop space machine}
 Algebras over the Barratt-Eccles operad $\sO$ in $\cat$ with $\sO(j)=\mathcal{E}{\SI_j}$ are symmetric monoidal categories with strict unit and strict associativity, which are also known as \emph{permutative categories} (see \cite{MayPerm2}).  By analogy, having an $E_\infty$-operad in $G\cat$ allows one to define \emph{genuine permutative $G$-categories} as algebras over it, and the classifying spaces of these turn out to be, after group completion, equivariant infinite loop spaces. This is carried out in the program started by Guillou and May in \cite{GM3}.

Of course, there are permutative categories, i.e., algebras over the Barratt-Eccles operad $\sO$, which are also $G$-categories, and \cite{GM3} reserves the name \emph{naive permutative $G$-categories} for those. The reason is that their classifying spaces are $G$-spaces, which are naive equivariant infinite loop spaces, i.e., they have deloopings with respect to all spheres with trivial $G$-action, but not necessarily with respect to representation spheres. We note that in this light, what we defined as symmetric monoidal $G$-categories are \emph{naive symmetric monoidal $G$-categories}.

 \begin{definition}
The operad $\sO_G$ in $G\cat$ defined by $\sO_G(j)= \cat(\tG, \mathcal{E}{\SI_j})$ is an $E_\infty$-$G$-operad. A \emph{genuine permutative $G$-category} is defined to be an  $\sO_G$-algebra.
\end{definition} 

If we take any \emph{naive permutative category} $\sC$, i.e., a permutative category with a $G$-action,  since it is   an algebra over the Barratt-Eccles operad $\sO$ with $\sO(j)=\mathcal{E}{\SI_j}$, there are maps $$\sO(j)\times \sC^j \to \sC$$ compatible with the operad structure maps. Since $\cat(\tG, -)$ is a product preserving functor, these maps yield maps  $$\cat(\tG, \mathcal{E}{\SI_j})\times \cat(\tG, \sC)^j \to \cat(\tG, \sC),$$ and all the necessary diagrams still commute, so $\cat(\tG, \sC)$ is a genuine permutative category, and surprisingly, the only examples  of genuine permutative categories we know arise in this way. 

\begin{example}
Recall  \myref{glr} of the category $\sGL(R)$. It is a skeleton of the category of finitely generated free $R$-modules $\sF(R)$. The category $\sGL(R)$ is permutative under direct sum of modules and block sum of matrices $\oplus: GL_n(R) \times GL_m(R) \rightarrow GL_{n+m}(R)$, since associativity and the unit are strict and commutativity holds only up to isomorphism (reordering of the basis elements by conjugation).  
It is a naive permutative $G$-category with trivial $G$-action on objects and entrywise $G$-action on matrices. Therefore the category $\cat(\tG, \sGL(R))$ is a genuine permutative $G$-category. 
\end{example}

The original May infinite loop space machine, which we will denote by $\bK$, was developed in \cite{MayGeo}; it takes as input a permutative category and produces $\OM$-spectra with zeroth space the group completion of the classifying space of the input permutative category. An equivariant version of May's operadic infinite loop space machine is developed in \cite{GM3}\footnote{On the $G$-space level, operadic infinite loop space theory was first developed in unpublished work of Costenoble, Haushild, May, and Waner in the early 1980's.}. It takes as input an $\mathscr{O}_G$-category $\C$, i.e., a genuine permutative $G$-category, and produces a genuine orthogonal $\OM$-$G$-spectrum with zeroth space the equivariant group completion of $B\C$. We give a brief overview of the machine. As explained in \cite{GM3}, we need to use not only an $E_\infty$ operad $\mathscr{C}_G$ in $G\Top$ (such as $B\sO_G$), but also the Steiner operads $\mathscr{K}_V$ indexed over finite dimensional subspaces of a complete $G$-universe $U$, because these act on $V$-fold loop spaces. These operads are described in detail in \cite[Appendix]{GM3}. Intuitively, they are a generalization of the little disks operad, which is compatible with suspension: instead of considering a tuple of embeddings of $V$ into $V$, one considers a tuple of paths of embeddings of  $V$ into $V$, which at time 0 are the identity and at time 1 are disjoint. So the picture of an element in the Steiner operad would look like a cylinder with $V$ at one end and en element of the little disks operad at the other end.   We define the product operad $$\mathscr{C}_V=\mathscr{C}_G\times \mathscr{K}_V.$$ 
A $\mathscr{C}_G$-space can be viewed as an $\mathscr{C}_V$-space for any $V$, and this has the advantage that $\C_V$ acts on $V$-fold loop spaces via its projection onto $\mathscr{K}_V$. Let ${\bf C}_V$  be the monad of based $G$-spaces associated to the operad $\mathscr{C}_V$.

For a genuine permutative $G$-category $\sA$, the orthogonal $G$-spectrum $\bK_G(\sA)$ has spaces given by the monadic bar constructions
$$ \bK_G\sA(V)=B(\Sigma^V, {\bf C}_V, B\sA).$$ The structure maps for $V\subset W$ are given by $$\Sigma^{W-V} B(\Sigma^V, {\bf C}_V, B\sA)\cong B(\Sigma^W, {\bf C}_V, B\sA) \rightarrow B(\Sigma^W, {\bf C}_W, B\sA).$$

\begin{theorem}[\cite{GM3}]
For a genuine permutative $G$-category $\sA$, the spectrum  $\bK_G\sA$ is a genuine $\OM$-$G$-spectrum and there is a group completion map $B\sA\ra (\bK_G\sA)(0)$.
\end{theorem}

The essential formal properties of the machine, which we will need, are the following theorems from \cite{GM3}.

\begin{thm}[\cite{GM3}]\mylabel{prop1} Let $\sA$ and $\sB$ be $\sO_G$-categories. Then the map 
\[ \bK_G(\sA\times \sB) \rtarr \bK_G \sA \times \bK_G\sB \]
induced by the projections is a weak equivalence of $G$-spectra.
\end{thm}
\begin{thm}[\cite{GM3}]\mylabel{prop2} For $\sO_G$-categories $\sA$, there is a natural weak equivalence of spectra
$$\bK(\sA^G) \rtarr (\bK_G \sA)^G.$$
\end{thm}

The inclusion $\io\colon \sO\rtarr \sO_G$ induces a forgetful functor 
$\io^*$ from genuine to naive permutative $G$-categories. Also, we have a forgetful functor $i^\ast$ from genuine to naive $G$-spectra.

\begin{thm}[\cite{GM3}]\mylabel{prop3} For $\sO_G$-categories $\sA$, there is a natural
weak equivalence of naive $G$-spectra $\bK \io^*\sA \rtarr i^*\bK_G\sA$.\end{thm}

By  \myref{skeleton equiv}, since $\sF(F)=\sP(F)$, the $K$-theory space of a field $F$ with $G$-action is the equivariant group completion of $B\cat(\tG, \sGL(F))$, and as we have seen in the example above, $\cat(\tG, \mathscr{GL}(F))$ is a genuine permutative $G$-category. Therefore, we can define $$\KK_G(F)=\bK_G(\cat(\tG,\mathscr{GL}(F))).$$

Nonequivariantly, it is well known that using a construction of MacLane from \cite{maclane-assoc}, any symmetric monoidal category $\C$ can be strictified to an equivalent permutative category $\C^{str}$, and therefore we can apply the nonequivariant infinite loop space machine $\bK$ to a symmetric monoidal category by implicitly doing this replacement first. The category $\C^{str}$ has objects given by strings $(c_1, \dots, c_n)$ of objects in $\C$, and morphisms $$(c_1,\dots, c_n)\ra (d_1, \dots, d_m)$$ given by morphisms $$c_1\oplus \dots \oplus c_n\ra d_1\oplus\dots \oplus d_m$$ in $\C$. The symmetric monoidal structure is given by concatenation and the identity is given by the empty string $()$. 

 This carries through equivariantly:  if $\C$ is a symmetric monoidal $G$-category, then $\C^{str}$ is naturally also a symmetric monoidal $G$-category, with $G$-action given on objects by $g(c_1, \dots, c_n)=(gc_1, \dots, gc_n)$. Since $G$ commutes with $\oplus$, 
we can define the action on morphisms by $$g\big((c_1,\dots, c_n)\xra{f} (d_1, \dots, d_m)\big)= (gc_1, \dots, gc_n) \xra{gf} (gd_1, \dots, gd_m).$$
 
\noindent It is not hard to see the inverse functors in the  equivalence $C\simeq C^{str}$ are $G$-equivariant. Therefore, given a symmetric monoidal $G$-category $\C$, the naive permutative $G$-category $\C^{str}$ is $G$-equivalent to it.  The upshot is that using this strictification implicitly, we can use the operadic machine on symmetric monoidal $G$-categories. 

\begin{rmk} Recall that when $\C$ is a symmetric monoidal category with $G$ that only preserves the symmetric monoidal structure up to coherent isomorphism, we have shown in Section 5.1., that the symmetric monoidal structure on $\cat(\tG, \C)$ is preserved by the $G$-action on the nose. So we by the the above construction the symmetric monoidal $G$-category $\cat(\tG, \C)$ is monoidally $G$-equivalent to the naive permutative category $\cat(\tG, \C)^{str}$. By using this equivalence and \myref{idem}, we get that $\cat(\tG, \sC)$ is $G$-equivalent to the genuine permutative category $\cat(\tG, \cat(\tG, \C)^{str})$. \end{rmk}

We give the following definition for all $G$-rings.
 
 \begin{definition}\mylabel{spectrum def}
 We define the equivariant algebraic $K$-theory spectrum of $R$ as
$$\KK_G(R)= \bK_G\big( \cat(\tG, \iso\sP(R))\big),$$ with the understanding that we have replaced the input symmetric monoidal $G$-category with an equivalent genuine permutative $G$-category. \end{definition}

We note that since the zeroth space $\KK_G(R)(0)$ is the group completion of $B \cat(\tG, \iso\sP(R))$, the spectrum  $\KK_G(R)$ does indeed give deloopings of the $K$-theory space $K_G(R)$ we had already defined. Alternatively, we can use the equivariant Segal machine for delooping this space, which we address in the next section.

\subsubsection{Equivariant Segal infinite loop space machine} 
Segal developed an alternative delooping machine to the operadic May machine in the celebrated paper \cite{segal}, which we will denote as $\bS$. The input is a $\GA$-space, which is just a functor $$X\colon \sF\ra \Top, \ \ {\bf n} \mapsto X_n,$$ where $\sF$ is a skeleton of the category of based finite sets\footnote{The opposite of Segal's original category $\GA$ turns out to be just $\sF$.}. A $\GA$-space is \emph{special} if the map $\de\colon X_n\rtarr X_1^n$, induced by the projections $\de_i\colon \mathbf{n}\rtarr \mathbf{1}$, is an equivalence. From a $\GA$-space, Segal produces a spectrum, and he shows that for a special $\GA$-space, the spectrum is $\OM$, with zeroth space the group completion of $X_1$. 

One can start with a $\GA$-category instead, i.e., a functor $\sF\ra \cat$ and define it to be special if the $\GA$-space obtained by applying the classifying space functor levelwise is a special $\GA$-space. Segal gives a construction of a special $\GA$-category $X$ from a symmetric monoidal category $\C$, with $ X_1\simeq \C$. Therefore, $\bS(\C)$, the spectrum obtained from the special $\GA$-space associated to the symmetric monoidal category $\C$, is $\OM$, with zeroth space the group completion of $B\C$.

Shimakawa has generalized Segal's machine in \cite{Shimakawa} to produce an orthogonal genuine $\OM$-$G$-spectrum starting from a special $\GA_G$-space. A $\GA_G$-space is a functor $$X\colon \sF_G\ra \Top_G,\ \  A\mapsto X(A),$$ where $\sF_G$ is the category of finite $G$-sets and $\Top_G$ is the category of $G$-spaces and nonequivariant based maps; $G$ acts by conjugation on morphism sets. For any $A\in \sF_G$, we have a projection $\de_a\colon A\ra {\bf 1}$, which sends all the nonbasepoint elements of $A$ to $1$ and the basepoint to $0$. A $\GA_G$-space is special if the map $\de_A\colon X(A)\ra \Map(A, X_1)$ induced by these projections is a $G$-equivalence for all $A\in \sF_G$. We note that this map turns out to be a $G$-map even though the individual maps $\de_a$ are generally not $G$-maps.

Given a $\GA_G$-space $X$, Shimakawa constructs a spectrum $\bS_GX$ with $V$th space given by the two-sided bar construction $B((S^V)^\bullet, \sF_G, X)$, where $(S^V)^\bullet$ is the contravariant functor $\sF_G\ra \Top_G$ defined on objects as $A\mapsto \Map(A, S^V)$. It is not hard to see that there are structure maps 
$$S^W \wedge B((S^V)^\bullet, \sF_G, X)\ra B((S^{V\oplus W})^\bullet, \sF_G, X).$$

The following is the main theorem in \cite{Shimakawa}.

\begin{theorem}[\cite{Shimakawa}]
For a special $\GA_G$-space $X$, the spectrum  $\bS_GX$ is a genuine $\OM$-$G$-spectrum, for which $X_1\simeq (\bS_GX)(0)$ if and only if $X_1$ is grouplike. \end{theorem}

Essential to our applications is that in general there is a group completion map  $X_1\ra (\bS_GX)(0)$, which Shimakawa does not prove, but we fill this gap in \cite{MMO}.

A $\GA_G$-category is a functor $\sF_G\to \cat_G$, where $\cat_G$ is the category of $G$-categories and nonequivariant functors. It is special if the $\GA_G$-space obtained by applying the classifying space functor levelwise is special. Shimakawa generalizes Segal's combinatorial way of constructing a $\GA$-category from a symmetric monoidal category to construct a $\GA_G$-category from a symmetric monoidal $G$-category $\C$. This $\GA_G$-category is not necessarily special, but Shimakawa shows that replacing $\C$ by the symmetric monoidal $G$-category $\cat(\tG, \C)$ does yield a special $\GA_G$-category, and therefore, $\bS_G(\cat(\tG, \C))$, the machine applied to the special $\GA_G$-category obtained from the symmetric monoidal $G$-category $\cat(\tG, \C)$, is a genuine orthogonal $\OM$-$G$-spectrum with zeroth space the group completion of $B\cat(\tG, \C)$.

In  \cite{MMO}, we prove that the two equivariant delooping machines agree; in particular, we have the following result.
\begin{theorem}\cite{MMO}
For a symmetric monoidal $G$-category $\C$ we have an equivalence of orthogonal $G$-spectra $\bK_G(\cat(\tG, \C))\simeq \bS_G(\cat(\tG, \C))$.
\end{theorem}

\begin{cor}
There is an equivalence of orthogonal  $\OM$-$G$-spectra $$\KK_G(R)\simeq \bS_G\big(\cat(\tG, \iso\sP(R))\big).$$
\end{cor}

\subsection{Functoriality of $\KK_G$}
Now we address functoriality of the construction. Even nonequivariantly, the assignment $R\mapsto \sP(R)$ is not a functor, but just a pseudo functor, since composition is not preserved strictly. One way to rectify any pseudo functor landing in $\cat$ to an actual functor is using Street's first construction from \cite{Street}. This generalizes equivariantly to strictify a pseudo functor landing in $G\cat$ to an actual functor. We will tacitly assume this strictification from now on, and address the new subtleties that arise and are specific to the equivariant case. The issue that arises is that for a $G$-map of $G$-rings $R\ra S$, the functor $\sP(R)\ra\sP(S)$ is \emph{not} equivariant, so it is not a morphism in $G\Cat$. So for sure the assignment $R\mapsto \sP(R)$ is not a functor or even a pseudo functor. We show that this gets rectified after applying $\cat(\tG, -)$.

\begin{theorem}
The assigment $R\mapsto \KK_G(R)$ is a functor from the category of $G$-rings and $G$-maps to  genuine orthogonal $G$-spectra. 
\end{theorem}

\begin{proof}
The equivariant infinite loop space machine $\bK_G$ is a functor from the category of genuine permutative $G$-categories and $G$-maps between them to the category of genuine orthogonal $G$-spectra. Thus it suffices to show that having a map of $G$-rings $R\ra S$ yields an \emph{equivariant} map  $\cat(\tG, \sP(R))\ra \cat(\tG, \sP(S))$.

Suppose $f\colon R\ra S$ is a $G$-map of $G$-rings, and consider the functor $\sP(R)\ra \sP(S)$ defined as $M\mapsto  M\otimes_R S$. 
Note that certainly $gM\otimes_R S\neq g(M\otimes_R S)$ since the scalar multiplication is different in the two modules; however we go on to define an isomorphism $$gM\otimes_R S\cong g(M\otimes_R S).$$
 
Recall that in $gM$, the scalar multiplication is defined as $r\cdot_g m={}^gr m$, where ${}^g(-)$ denotes the action of $g$ on $R$. Define $$gM\otimes_R S\ra g(M\otimes_R S), \ \ \ m\otimes s\mapsto m\otimes {}^gs.$$ 
First of all, we use the assumption that $f$ is a $G$-map to show that this assignment is well-defined. Note that for $t\in R$, we have the following identification in $gM\ox_R S$:
$$m\otimes ts \sim f(t)\cdot_g m\ox s  = {}^g\!f(t) m \ox s,$$
Now $$m\otimes ts \mapsto m\ox {}^g(ts)= m\ox {}^gt\ \!{}^g\!s \sim f({}^gt) m\ox {}^gs,$$
and $${}^gf(t)m\ox s\mapsto {}^gf(t)m\ox {}^gs,$$ but these are equal since $f({}^gt)={}^gf(t).$

We check next that the assignment is $S$-linear: for $t\in S$, 
$$t(m\otimes s)=m\otimes ts \mapsto m\otimes {}^g (ts)=m \otimes {}^gt\ \!{}^g\!s=t\cdot_g(m\otimes {}^gs).$$

Similarly, we can check that the inverse map $$g(M\ox_R S)\ra gM\ox_R S, \ \ \ m\ox s\mapsto m\ox {}^{(g^{-1})}s$$ is well-defined and $S$-linear, so that we have the claimed isomorphism.

It is not hard to see that these isomorphisms make  the functor $-\otimes_R S$ pseudo equivariant, and  \myref{pseudo equiv} provides the desired $G$-map $$\cat(\tG, \sP(R))\ra \cat(\tG, \sP(S)),$$ which in turn gives a map of genuine $G$-spectra $$\KK_G(R)\ra \KK_G(S)$$ by the functoriality of the equivariant infinite loop space machine $\bK_G$. 
\end{proof}


\subsection{Properties of the equivariant algebraic $K$-theory spectrum $\KK_G(R)$}

Now we can exploit the results that we have proved in section 4 about the homotopy fixed points of module categories. An immediate consequence \myref{prop2} and \myref{equiv chain proj} is the following theorem, which says that we recover the classical nonequivariant $K$-theory of twisted group rings as the fixed points of our construction.

\begin{theorem}\mylabel{fixed point spectra}
If $H\subseteq G$ and $|H|^{-1}\in R$, there is an equivalence of spectra $$\KK_G(R)^H\simeq\KK(R_H[H]).$$
\end{theorem}

By \myref{morita proj}, we immediately get that the equivariant algebraic $K$-theory of $G$-rings is equivariantly Morita invariant.

\begin{proposition}
If $R$ and $S$ are equivariantly Morita equivalent, then there is a $G$-equivalence $$\KK_G(R)\simeq \KK_G(S).$$
\end{proposition}

 \section{Equivariant $K$-theory of Galois extensions}\label{galois}
 
 The algebraic $K$-theory of Galois extensions behaves particularly nicely as a result of faithfully flat descent and the fact that for $G$-Galois extensions the category of descent data has an interpretation in terms of modules with semilinear $G$-action.  
 
 \subsection{Galois extensions of rings}
Galois extensions of rings have been  introduced and first studied by Auslander and Golman in  \cite{galois}. For a ring extension $R\ra S$, let $\Aut_R(S)$ be the group of automorphisms of $S$ fixing $R$. We recall the definition.

\begin{definition}
 Let $R\ra S$ be a faithfully flat ring extension and suppose that $G$ is a finite subgroup of $\Aut_R(S)$. The extension $R\ra S$ is Galois with Galois group $G$ if the map
 \begin{equation}\mylabel{Galois condition}
 \ga \colon S\otimes_R S \ra \prod_G S, \ \ \  a\ox b\mapsto ((g\cdot a) b)_{g\in G}.
 \end{equation}
 is an $S$-algebra isomorphism.
 \end{definition}

It is an easy exercise to see that
  $R=S^G$. 
The wonderful fact about Galois extensions is that if $R\ra S$ is a Galois ring extension with Galois group $G$, then the category of $S$-modules with semilinear $G$-action is equivalent to  the category of $S$-modules with descent data, and in turn, by faithfully flat descent this is equivalent to the category of modules over $R$.

In the proof of \myref{equiv chain} we showed that for a $G$-ring $S$, the category of $S$-modules with semilinear $G$-action is equivalent to the homotopy fixed point category $\cat(\tG, \Mod(S))^G$. Thus we have the following proposition.

\begin{proposition}\mylabel{galois descent} 
 Suppose $R\ra S$ is a Galois ring extension with Galois group $G$. Then there is an equivalence of categories $$\Mod(R)\simeq \cat(\tG, \Mod(S))^G.$$
\end{proposition}

Suppose $R\ra S$ is faithfully flat. Then an $R$-module $M$ is finitely generated  projective if and only if the $S$-module $M\ox_R S$ is finitely generated projective (see \cite[Prop. 2.12.]{fields}). Therefore, the equivalence of categories mentioned above restricts to an equivalence between the corresponding categories of finitely generated projective modules, and we obtain the following analogue of \myref{galois descent}.

\begin{proposition}\mylabel{proj galois descent}
 Suppose $R\ra S$ is a Galois ring extension with Galois group $G$. Then there is an equivalence of categories $$\sP(R)\simeq \cat(\tG, \sP(S))^G.$$
\end{proposition}

This leads to the following theorem, which says that for a $G$-Galois extension $R\ra S$, the $G$-fixed point spectrum of the $G$-equivariant $K$-theory of $S$ is the same as the nonequivariant $K$-theory spectrum of the fixed ring $S^G=R$.

\begin{theorem}\mylabel{galois result}
Let $R\ra S$ be a Galois extension of rings with Galois group $G$. Then there is an equivalence of orthogonal spectra $$\KK_G(S)^G\simeq \KK(R).$$
\end{theorem}

\begin{proof}
By \myref{proj galois descent}, we have an equivalence of categories $$\sP(R)\simeq \cat(\tG, \sP(S))^G.$$ By \myref{prop2}, we have 
$$\bK_G(\cat(\tG, \sP(S)))^G \simeq \bK(\cat(\tG, \sP(S))^G).$$ Therefore,
$$\KK_G(S)^G\simeq \KK(R).$$
\end{proof}

\begin{example}\mylabel{galois example}
For any finite $G$-Galois extension of fields $E/F$,  $\KK_G(E)^G\simeq \KK(F).$ In particular, this recovers $\KK(\bQ)$ as the fixed point spectrum of the genuine equivariant $K$-theory spectrum of any finite Galois extension of $\bQ$.
\end{example}

\begin{example}
For any ring $R$, the diagonal map $R\ra R\times R$ is a Galois extension with group $\bZ/2\bZ$, where the nontrivial element acts on $R\times R$ by interchanging the factors. Thus, $$\KK_{\bZ/2\bZ}(R\times R)^{\bZ/2\bZ}\simeq \KK(R).$$
\end{example}

 \subsection{Strong form of Hilbert's theorem 90}\mylabel{hilbert 90}
 
  As an accidental corollary of our interpretations of homotopy fixed point categories of modules, we obtain a new proof of Serre's generalization of Hilbert's theorem 90. The original Hilbert 90 theorem states that for a Galois extension $E/F$ with Galois group $G$, the cohomology group $H^1(G, E^\times)$ is trivial. This theorem is reproved in \cite{sga} using faithfully flat descent. In the same spirit, we use \myref{galois descent} to give an alternative proof of the generalization of Hilbert's theorem 90, which is due to Serre:
 \begin{theorem}[Serre]\mylabel{serre}
 The nonabelian cohomology $H^1(G, GL_n(E))$ is trivial.
 \end{theorem}
 
 \begin{proof}
 Faithfully flat descent gives an equivalence of categories between the category of $F$-modules and the category of descent data, which in turn is equivalent to the category of $E$-modules with semilinear $G$-action. From the proof of \myref{equiv chain}, the latter is just $\cat(\tG, \iso \Mod(E))^G$. Note that $\coprod_n GL_n(E)$ is a skeleton of  $\iso \Mod(E)$, and by \myref{homotopy fixed point weak equiv} we have a weak equivalence
 $$\cat(\tG, \iso \Mod(E))^G \simeq \cat(\tG, \coprod_n GL_n(E))^G.$$ Using that $\coprod_n GL_n(F)$ is a skeleton for $\iso \Mod(F)$  and  the description of the fixed points of $\cat(\tG, GL_n(E))$ given in \myref{crossedhom},  we have $$\coprod_n GL_n(F)  \simeq \coprod_n \cat_\times(G, GL_n(E)),$$ where $\cat_\times(G, GL_n(E))$ is the crossed functor category defined in \myref{crossedhom}. 
 Therefore, for any summand in the coproduct, there is only one isomorphism class of objects. Now we note that the isomorphism set of objects in $\cat_\times(G, GL_n(E))$ is precisely the first nonabelian cohomology set $H^1(G, GL_n(E))$.
 \end{proof}

 This gives the following result, which we could have used directly to conclude that for a finite  Galois extension of fields $E/F$ with Galois group $G$ we have an equivalence $\KK_G(E)^G\simeq \KK(F).$

\begin{proposition}\mylabel{field equiv}
There is a symmetric monoidal weak $G$-equivalence $$\io \colon \sGL(E)\ra \cat(\tG, \sGL(E)).$$
\end{proposition} 

\begin{proof}
By \myref{H1 trivial}, we have that the map  $$\io \colon \sGL(E)^H\ra \cat(\tG, \sGL(E))^H$$ is an equivalence precisely when  $H^1(G, GL_n(E))$ is trivial, and this is true in this case by \myref{serre}. Now also note that $$\sGL(E)^H= \sGL(E^H).$$ This completes the proof.
\end{proof}

\begin{remark}
Let $k$ be a field with separable closure $\bar{k}$ and absolute Galois group $\Gal(\bar{k}/k)$, which acts continuously on the discrete group $\bar{k}^\times$, or more generally, on $GL_n(\bar{k}).$ We have that $H^1(\Gal(\bar{k}/k), GL_n(\bar{k})) = \text{colim }H^1( \Gal(L/k), GL_n(\bar{k})),$ where the inclusions of $H^1$ groups in the colimit come from the inclusions of cocycles (continuous crossed homomorphisms in this case).  We get an equivalence analogous to the one in  \myref{field equiv}:
$$\mathscr{G}\mathscr{L}(k) \xra{\simeq} \cat(\Gal(\bar{k}/ k), \mathscr{G}\mathscr{L}(\bar{k}))^{\Gal(\bar{k},k)}.$$

However, the equivariant infinite loop space machines, as currently developed, do not apply to profinite groups, so we cannot pass from this statement to a spectrum level statement. We do hope, though, that in future work we will be able to generalize the delooping machines to profinite groups. 

\end{remark}

\subsection{Quillen-Lichtenbaum formulation}\mylabel{section:QL}

Let $E/F$ be a finite  Galois extension of fields with Galois group $G=\Gal(E/F)$ on $E$.  Since $G$ acts on $E$, it acts by functoriality on the spectrum $\KK E$, so $\KK E$ is a naive $G$-spectrum. The fixed points of this naive $G$-spectrum are easily seen to be $(\KK E)^G\simeq \KK F$. The initial Quillen-Lichtenbaum conjecture was that the map of spectra $$\KK F\simeq  \KK E^G\to \KK E^{hG}$$ is an equivalence, where $\KK E^{hG}$ denotes the homotopy fixed points of the naive $G$-spectrum $\KK E$. However, low dimensional examples disprove this conjecture as stated even after $p$-adic completion for a prime $p$ \cite{mitchell}. Thomason showed that under some technical hypotheses this map  becomes an equivalence only after reducing mod a prime power and inverting the Bott element \cite{Thomason}.\footnote{The Quillen-Lichtenbaum conjecture has been refined to the statement that the map from fixed points to homotopy fixed points of the separable closure of a field for the action of the absolute Galois group is an equivalence mod $p$ on a connected cover, and was resolved in the two primary case by Rosenschon and  \O stv\ae r in \cite{rosenschorn}. The result at the prime 2 was already known by work of Levine for fields of finite cohomological dimension \cite{levine}, and it is known that the proof of the Bloch-Kato would extend the result  for such fields to other primes.} 

One might ask the same question of the map from fixed points to homotopy fixed points of the genuine $K$-theory $G$-spectrum of $\KK_G(E)$ and hope that it becomes an equivalence there, but we show that that map is equivalent to the one in terms of naive spectra from the original Quillen-Lichtenbaum conjecture. In \cite{jeremiah}, J. Heller and J. Hornbostel  give a definition of a genuine equivariant algebraic $K$-theory spectrum by constructing a certain special $\sF_G$-space, and they show that for their construction the map from fixed points to homotopy fixed points  becomes an equivalence when a lift of the Bott element is inverted. We believe that our definitions could suitably be related and then the result of this section could be viewed as a reflection of their theorem.

The definition of fixed points and homotopy fixed points of a genuine $G$-spectrum is just as the fixed points, or homotopy fixed points, respectively, of the underlying naive $G$-spectrum. However, for an arbitrary ring with $G$-action, the underlying naive $G$-spectrum of $\KK_G(R)$ is not necessarily equivalent to the naive $G$-spectrum $\KK(R)$. 
We review the precise definition of homotopy fixed points of a fibrant genuine orthogonal $G$-spectrum $X$. Note that a fibrant spectrum is an $\OM$-$G$-spectrum.
\begin{definition}
Let $X$ be a fibrant genuine  orthogonal $G$-spectrum. Then $X^{hG}$ is defined to be the fixed point spectrum $$\Map_\ast(EG_+, X)^G= (i^\ast\Map_\ast(EG_+, X))^G\simeq \Map_\ast(EG_+, i^\ast X)^G,$$ where $i^\ast$ is the forgetful functor from genuine to naive $G$-spectra.
\end{definition}

Just as for $G$-spaces and naive $G$-spectra, we have a natural map $X^G\to X^{hG},$ induced by the projection $EG_+\ra S^0$.

\begin{proposition}
Let $E/F$ be a finite Galois extension with Galois group $G$. The map from fixed points to homotopy fixed points of genuine $G$-spectra $$\KK_G(E)^G \to \KK_G(E)^{hG}$$ is equivalent to the map from fixed points to homotopy fixed points of naive $G$-spectra $$\KK F \to \KK^{hG}$$ in the sense that the following diagram commutes and the vertical maps are weak equivalences
\[\xymatrix{
\KK F\ar[rr] \ar[dd]^\simeq && \KK E^{hG} \ar[dd]^\simeq\\
&&&\\
\KK_G(E)^G \ar[rr] && \KK_G(E)^{hG}
}\]
\end{proposition}

From \myref{galois example}, we have that $\KK F\simeq \KK_G(E)^G$. We show in the next proposition that $\KK_G(E)^{hG} \simeq \KK(E)^{hG}$, where on the left hand side we are taking homotopy fixed points of a genuine $G$-spectrum, and on the right hand side we are taking homotopy fixed points of a naive $G$-spectrum. As pointed out above, this amounts to comparing the underlying naive $G$-spectra, and noting that in the case of a Galois extension they are equivalent. This proves the above proposition, and thus we recover the initial form of the Quillen-Lichtenbaum conjecture as a statement about genuine $G$-spectra. 

\begin{lem}
For a Galois extension $E/F$ with Galois group $G$, the homotopy fixed points of the naive $G$-spectrum $\KK E$ and the homotopy fixed points of the genuine $G$-spectrum $\KK_G(E)$ are equivalent.
\end{lem}
\begin{proof}

We will show that we have an equivalence of naive $G$-spectra $i^\ast \KK_G(E)\simeq \KK E$, which will imply the result. Recall that $\KK E$ is defined as the $K$ theory of the naive permutative $G$-category $\sG\sL(E)$, while $\KK_G(E)$ is the equivariant algebraic $K$-theory of the genuine permutative $G$-category $\cat(\tG, \sG\sL(E))$. We have a map
$$\KK(\sG\sL(E))\ra \KK(i^\ast(\cat(\tG, \sG\sL(E)))\xra{\simeq} i^\ast \KK_G(\cat(\tG, \sG\sL(E))),$$ where the second map is shown to be an equivalence in \cite{GM3}.

By  \myref{field equiv}, there is a symmetric monoidal weak $G$-equivalence $$\sG\sL(E)\simeq i^\ast(\cat(\tG, \sG\sL(E)),$$ so the first map is also an equivalence.

\end{proof}

\subsection{Carlsson's assembly map from the equivariant perspective}

There has been a long standing program initiated and lead by G. Carlsson of studying the $K$-theory of fields motivated by the concept of descent and the Quillen-Lichtenbaum conjecture. 
Suppose that $E/F$ is a Galois extension with Galois group $G$. We can consider the \emph{assembly map} induced in $K$-theory by  extension of scalars \begin{equation}\mylabel{assembly 1} \mathrm{Rep}_F[G] \xra{E\ox_F -}  V^G(E),\end{equation}
 from the category of continuous finite dimensional $G$-representations in $F$, which is denoted by $\Rep_F[G]$ in \cite{carlssonconj}, to the category of finite dimensional $E$-vector spaces with semilinear $G$-action, or equivalently, the category of $E$-vector spaces with descent data for the faithfully flat extension $E/F$, denoted by $V^G(E)$ in \cite{carlssonconj}, and which is equivalent to the category $\Vect(F)$ of finite dimensional $F$-vector spaces by \myref{galois descent}. Since all these categories are nonequivariant symmetric monoidal categories, their $K$-theory spectra are defined by using standard nonequivariant infinite loop space machines such as the May \cite{MayGeo} or the Segal machine \cite{segal}. 

Carlsson conjectured that for the Galois extension $\bar{F}/F$ with absolute Galois group $G$ and an algebraically closed subfield $k\hookrightarrow F$ the composite  \begin{equation}\mylabel{assembly 2}\mathrm{Rep}_k[G] \ra \mathrm{Rep}_F[G] \xra{\bar{F}\ox_F -}  V^G(E)\simeq \Vect(F).\end{equation}
induces an equivalence on $K$-theory after derived completion, i.e.,
$$\KK(\mathrm{Rep}_k[G])^\wedge_{\al_p} \to \KK(F)^\wedge_{\al_p} \simeq \KK(F)^\wedge_p.$$ 
 Carlsson defines \emph{the derived completion} of a ring spectrum in \cite{derivedcompletion} and shows that for $\KK F$ the derived completion agrees with the Bousfield-Kan $p$ completion, which accounts for the identification made above on the right hand side. This conjecture has recently been proved by G. Carlsson for pro-$l$ absolute Galois groups in \cite{carlsson1} and \cite{carlsson2}. Attacking the problem from the perspective of full-fledged equivariant spectra might eventually shed some light on the general case. C. Barwick has announced  such a proof from an $\oo$-categorical point of view using spectral Mackey functors. 

The key to any equivariant point of view is, of course, to interpret the assembly map as the fixed point map of an equivariant map between genuine $G$-spectra. We show how to construct such an equivariant map. On the categorical level our definition makes sense for the separable Galois extension of a field and the profinite Galois group; however, we only know how to obtain a map of genuine $G$-spectra for a finite Galois extension at the moment, because of the limitations of the equivariant infinite loop space machines. It is very easy to see that the source and target of the assembly map are instances of fixed points of our construction of equivariant algebraic $K$-theory, but it is non-trivial to see that there is an equivariant map which restricts to the assembly map on fixed points. This relies on the fact that our construction of equivariant algebraic $K$-theory turns pseudo-equivariant maps to on the nose equivariant maps.

\begin{proposition}Suppose $E/F$ is a finite Galois extension with Galois group $G$.  There is a $G$-map of genuine $G$-spectra 
$$\KK_G(F)\to \KK_G(E),$$ which, on $G$-fixed points, restricts to the assembly map $$\KK(\mathrm{Rep}_F[G])\to \KK(V^G(E)).$$ \end{proposition}

\begin{proof}
It is easy to identify the source and target. From \myref{equiv chain}, the category $V^G(E)$ of $E$-vector spaces with semilinear $G$-action is isomorphic to  $\cat(\tG, \Vect(F))^G$. As we have remarked before, if $G$ is acting trivially on $F$, $\cat(\tG, \Vect(F))^G\cong  \cat(G, \Vect(F))$, the category of $F$-vector spaces with $G$-linear action, i.e., the category of continuous $G$-representations over $F$. Thus $\mathrm{Rep}_F[G]\cong \cat(\tG, \Vect(F))^G$.
Therefore, the map from equation (\ref{assembly 1}), translates to a map $$\cat(\tG, \Vect(F))^G \to \cat(\tG, \Vect(E))^G.$$
\noindent It remains to show that we indeed have an equivariant map  $$\cat(\tG, \Vect(F)) \to \cat(\tG, \Vect(E)),$$ inducing this map on fixed points.

Note that the extension of scalars map $\Vect(F)\xra{E\ox_F -} \Vect(E)$ is not a $G$-map. The action of $G$ on $\Vect(F)$ is trivial; however, for $V\in \Vect(F)$, the object $E\otimes_F V$ is not $G$-fixed: $$E\otimes_F V \neq g(E\otimes_F V)$$ since they have different scalar multiplication. However, we have shown that the extension of scalars map is pseudo equivariant, thus by \myref{pseudo equiv} this induces an equivariant map  $\cat(\tG, \Vect(F))\ra \cat(\tG, \Vect(E)),$ which on application of the equivariant infinite loop space machine $\KK_G$ yields a $G$-map $$\KK_G(F)\to \KK_G(E),$$ which restricts to the assembly map on fixed points.


\end{proof}

\section{Equivariant algebraic $K$-theory of topological rings}
We describe how our construction of equivariant algebraic $K$-theory recovers the connective covers of the the well-known equivariant topological real and complex $K$-theories $KU_G$ and $KO_G$, defined in \cite{equivktheory} and Atiyah's Real $K$-theory $KR$, defined in \cite{atiyah}. When Atiyah introduced  $KR$, he described it as a mixture of real $K$-theory $KO$ and equivariant topological $K$-theory $KU_G$ and $KO_G$.  We show that they all fit under the unifying framework of equivariant $K$-theory developed in this paper. 

In this section, any time we refer to topological $K$-theory, we mean the connective version.  We denote by $ku_G$, $ko_G$ and $kr$ the connective covers of $KU_G$, $KO_G$, and $KR$, respectively. Whereas the first two are well studied, the latter is not so well-known. A construction of the connective cover of $KR$ is given, for example, in \cite{dugger}.

Recall that as topological groups $GL_n(\bC)\simeq U_n$ and $GL_n(\bR)\simeq O_n$, and if one takes the topology into account when forming the bar construction, $BGL_n(\bC)$ and $BGL_n(\bR)$ are equivalent to the Grassmanians $BU_n$ and $BO_n$. We recall that the representing spaces for complex and real topological $K$-theory, namely $BU\times \bZ$ and $BO\times \bZ$, are the group completions of the topological monoids $\coprod BU_n$ and $\coprod BO_n$, respectively. Therefore,

 $$\KK^{top}(\bC)\simeq ku\ \text{   and     } \  \KK^{top}(\bR)\simeq ko,$$ where $\KK^{top}$ is algebraic $K$-theory for which the topology on the ring is taken into account when forming the bar construction.

We note that $ku_G$ and $ko_G$ are represented by the $G$-spaces which are the group completions of the monoids of   equivariant bundles corresponding to split extensions
\begin{equation}\mylabel{complex} 1\ra U_n\ra U_n\times G \ra G\ra 1\end{equation}
 and \begin{equation}\mylabel{real}1\ra O_n\ra O_n\times G \ra G\ra 1, \end{equation} respectively.

Consider the topological rings $\bC$ and $\bR$ with trivial $G$-action for any finite group $G$. Then $\KK_G(\bC)$ and $\KK_G(\bR)$ are genuine $\OM$-$G$-spectra with zeroth spaces given by the group completions of $\coprod B\cat(\tG,  GL_n(\bC))$, and $\coprod B\cat(\tG,  GL_n(\bR))$, respectively, where the topology of $GL_n(\bC)$, and $GL_n(\bR)$, respectively, is taken into account when forming the classifying space.  By \myref{bundle thm}, these are the monoids of classifying spaces of $(G,U_n)$-bundles, and $(G, O_n)$-bundles, respectively, under Whitney sum. Note that here it was crucial that in the hypotheses of \myref{bundle thm}, even though the group of equivariance $G$ has to be discrete or finite, the structure group of the bundle is allowed to be compact Lie. Therefore, we  obtain the following theorem, where $\KK_G^{top}$ is $G$-equivariant algebraic $K$-theory for which the topology of the ring is taken into account.

\begin{theorem}\mylabel{ku ko}
Consider  the topological rings $\bC$ and $\bR$ with trivial $G$-action for any finite group $G$. We have equivalences of connective $\OM$-$G$-spectra 
$$\KK^{top}_G(\bC)\simeq ku_G \text{   and     } \  \KK^{top}_G(\bR)\simeq ko_G.$$
\end{theorem}

In the definition of $KR$, the bundles corresponding to split exact sequences (\ref{complex}) are replaced by equivariant $(C_2, U_n\rtimes C_2)$-bundles corresponding to split exact sequences 
\begin{equation}\mylabel{kr} 1\ra U_n\ra U_n\rtimes C_2 \ra C_2 \ra 1,\end{equation}
where the cyclic group of order 2, $C_2$, acts on $U_n$ by complex conjugation.

Atiyah shows a ``Real" version of Bott periodicity, which gives that the representing space for $KR$ has deloopings with respect to  $C_2$-representations, and thus $KR$ represents a genuine $\OM$-$C_2$-spectrum. Of course, \cite{atiyah} does not mention spectra and instead states the result in terms of a periodic $RO(C_2$)-graded cohomology theory.
 
 The zeroth space of the connective spectrum $kr$ is the group completion of the topological $C_2$-monoid of $(C_2, U_n\rtimes C_2)$-bundles, which by \myref{bundle thm} and because the equivalence of topological groups $GL(\bC)\simeq U_n$ is $C_2$-equivariant, is equivalent to $\coprod B\cat(\tG, GL_n(\bC))$. Therefore, we get the following theorem. 

 \begin{theorem}\mylabel{kr}
 Let $\bC$ be the topological ring of complex numbers with conjugation action by $C_2$. Then there is an equivalence of connective $\OM$-$C_2$-spectra $$\KK^{top}_{C_2}(\bC)\simeq kr.$$
 \end{theorem}

  \bibliographystyle{amsalpha}
  \bibliography{../references}

\end{document}